\documentclass[11pt,twoside,reqno,centertags]{amsart}
\usepackage{amsmath,amsthm,amsfonts,amssymb}
\usepackage{color}
\advance\textheight by 6truemm
\newtheorem{Theorem}{Theorem}[section]
\newtheorem{Lemma}{Lemma}[section]
\newtheorem{theorem}[Theorem]{Theorem}

\newtheorem{proposition}{Proposition}[section]
\newtheorem{lemma}{Lemma}[section]
\newtheorem{remark}{Remark}[section]
\newtheorem{example}{Example}[section]
\newcommand{\R}{{\mathbb R}}
\newcommand{\eps}{\varepsilon}

\newcommand{\fzero}{{{\bar f}_0}}
\newcommand{\finfty}{{{\bar f}_\infty}}

\numberwithin{equation}{section}

\begin{document}

\title[Liouville theorems and universal estimates]{Liouville theorems and
universal estimates \\
for superlinear parabolic problems \\
without scale invariance} 
\author[Pavol Quittner and Philippe Souplet]{Pavol Quittner$^{(1)}$ and Philippe Souplet$^{(2)}$}

\thanks{$^{(1)}$Department of Applied Mathematics and Statistics, Comenius University
Mlynsk\'a dolina, 84248 Bratislava, Slovakia. Email: quittner@fmph.uniba.sk} 

\thanks{$^{(2)}$Universit\'e Sorbonne Paris Nord, CNRS UMR 7539, LAGA,
93430 Villetaneuse, France. Email: souplet@math.univ-paris13.fr}

\date{}

\begin{abstract}
We establish Liouville type theorems in the whole space and in a half-space
for parabolic problems without scale invariance.
To this end, we employ two methods, respectively based on
the corresponding elliptic Liouville type theorems and energy estimates for suitably rescaled problems,
and on reduction to a scalar equation by proportionality of components.

We then give applications of known and new Liouville type theorems to universal singularity and decay estimates
for non scale invariant parabolic equations and systems involving superlinear nonlinearities with regular variation.
To this end, we adapt methods from \cite{QS24} to parabolic problems.

\vskip 0.1cm

\end{abstract}

\maketitle

\section{Introduction}
\label{intro}

The present paper is concerned with Liouville type theorems
and universal estimates of solutions of parabolic systems of the form
\begin{equation} \label{eq-fU}
 U_t-\Delta U=f(U), 
\end{equation}
where $U=(u_1,\dots,u_N)$, $N\geq1$, and $f=(f_1,\dots,f_N):\R^N\to\R^N$ is a superlinear
function which need not be homogeneous, i.e.~system \eqref{eq-fU}
need not be scale invariant.
Unless explicitly stated otherwise, by a solution we mean a classical solution.

It is well known that in the scale invariant case, Liouville theorems
guaranteeing the nonexistence of nontrivial bounded nonnegative solutions of
\eqref{eq-fU} in $\R^n\times\R$ imply universal estimates of nonnegative solutions
to \eqref{eq-fU} 
in general domains $D:=\Omega\times(0,T)$, see \cite{PQS07b,Q21}. 
The class of nonnegative solutions can be replaced by suitable subclasses of nodal solutions, 
see \cite{Q22} and the references therein. 
Recently, it was shown in
\cite{S23} that Liouville theorems imply universal estimates
also for some scalar elliptic and parabolic problems without scale invariance,
but the required parabolic Liouville theorems have not been proved
in the full subcritical range, see \cite[Section~5]{S23}.   
The main aim of the present paper is to prove
optimal Liouville theorems for parabolic equations and systems without scale invariance 
and to show that they imply universal estimates also in the case of systems.
 
Let us first discuss the model problem \eqref{eq-fU} with $N=1$ and $f(U)=U^p$, $p>1$. 
In this case, problem \eqref{eq-fU} does not possess nonnegative nontrivial
solutions in $\R^n\times\R$ if (and only if) $p<p_S$, where 
$$
p_S:=\left\{
  \begin{aligned}
    &\frac{n+2}{n-2},&\quad &\text{if $n\ge 3$,}\\
    &\infty,&\quad &\text{if $n\in\{1,2\}$,}
  \end{aligned}\right.
$$
is the critical Sobolev exponent, see \cite{Q21}.
This Liouville theorem implies that any nonnegative solution of
\eqref{eq-fU} with $f(U)=U^p$, $p\in(1,p_S)$, in $D=\Omega\times(0,T)$ 
(where $\Omega$ is an arbitrary domain in $\R^n$ and $T\in(0,\infty]$)
satisfies  the estimate
\begin{equation} \label{UBp}
 U^{p-1}(x,t)\leq C(n,p)\,\hbox{\rm dist}^{-2}((x,t),\partial D), 
\end{equation}
where $\hbox{\rm dist}((x,t),\partial D)=\inf\{|x-y|+|t-s|^{1/2}:(y,s)\in\partial D\}$,
see \cite{PQS07b}.

If  $f\in C([0,\infty))$ satisfies 
\begin{equation} \label{fup}
\lim_{U\to\infty}\frac{f(U)}{U^{p}}=c\in(0,\infty), \qquad\hbox{where }\ p\in(1,p_S), 
\end{equation}
then \cite{Q21,PQS07b} imply that
any nonnegative solution of \eqref{eq-fU} in $D$ satisfies 
\begin{equation} \label{UBfp}
 U^{p-1}(x,t)\leq C(n,f)(1+\hbox{\rm dist}^{-2}((x,t),\partial D)).
\end{equation}
Estimate \eqref{UBfp} is weaker than \eqref{UBp}: In particular, it does not yield decay estimates
(for unbounded $D$).
If $\Omega$ is regular and $U=0$ on $\partial\Omega\times(0,T)$,
then the corresponding Liouville theorem
for bounded solutions of the equation $U_t-\Delta U=U^p$ 
in the half-space $\R^n_+\times\R$ satisfying $U=0$ on $\partial\R^n_+\times\R$
guarantees that
estimate \eqref{UBfp} can be improved to the estimate
\begin{equation} \label{UBfpOmega}
  U^{p-1}(x,t)\leq C(n,f,\Omega)(1+t^{-1}+(T-t)^{-1}),
\end{equation}
see \cite[Section~6]{PQS07b} for a more general case $f=f(x,t,U,\nabla U)$.

On the other hand,
\cite[Theorem~3.1 and Remark~1(i)]{S23} imply
that if $N=1$ and we 
replace assumption \eqref{fup} by the weaker assumption
\begin{equation} \label{frv}
 \lim_{\lambda\to\infty}\frac{f(\lambda U)}{f(\lambda)}=U^p \hbox{ for each }U>0,
\qquad\hbox{where }\ p\in(1,p_S),
\end{equation}
then we still have the following modification of estimate \eqref{UBfp}
\begin{equation} \label{UBflarge}
 \frac{f(U(x,t))}{U(x,t)}\leq C(n,f)(1+\hbox{\rm dist}^{-2}((x,t),\partial D)) \quad\hbox{ whenever }\ U(x,t)\geq1
\end{equation}
(and analogous modification is true in the case of estimate \eqref{UBfpOmega}).
In particular, if 
\begin{equation} \label{flog}
f(U)=|U|^{p-1}U\log^a(2+U^2),\quad\hbox{ where }\ a\in\R, \ p\in(1,p_S), 
\end{equation}
then $f$ satisfies \eqref{frv},
and inequality \eqref{UBflarge} in the case $\Omega=\R^n$ implies optimal blow-up rate
estimate of nonnegative solutions which blow up at $t=T$, see \cite[Remark~2]{S23}.
Such blow-up rate estimates (with $C$ depending also on $U$) have been obtained
in the case of the nonlinearity \eqref{flog}
even for sign-changing solutions in \cite{HZ22}, but the method in \cite{HZ22}
is very technical and it is not clear whether it can be applied to
the whole class of functions $f$ satisfying \eqref{frv}.

In the scalar case,
Liouville theorems proved in the present paper guarantee
that for a large class of nonlinearities $f$ satisfying
\eqref{frv}, 
\begin{equation} \label{frv2}
 \lim_{\lambda\to0+}\frac{f(\lambda U)}{f(\lambda)}=U^q \hbox{ for each }U>0,
\qquad\hbox{where }\ q\in(1,p_S),
\end{equation}
and some additional assumptions,
estimate \eqref{UBflarge} can be improved to the estimate
\begin{equation} \label{UBfrv}
 \frac{f(U(x,t))}{U(x,t)}\leq C(n,f)\,\hbox{\rm dist}^{-2}((x,t),\partial D),
\end{equation}
which coincides with estimate \eqref{UBp} if $f(U)=U^p$,
and in particular provides (universal) decay in time and space.
For instance, estimate \eqref{UBfrv} is true
in the case of \eqref{flog} with $a$ satisfying \eqref{HZa}, see Example~\ref{ex-HZ}.  

One of the main motivations of this paper is to prove such estimates also in the case of parabolic systems.
Namely, for a large class of $f:\R^N\to\R^N$ with $N>1$, as a consequence of our new Liouville theorems, 
we shall derive estimates \eqref{UBflarge} or \eqref{UBfrv} with the left-hand sides replaced by
${|f(U(x,t))|}/{|U(x,t)|}$, where $|\cdot|$ denotes the Euclidean norm in $\R^N$; see Theorem~\ref{thmUB} below.

Our main Liouville theorems are Theorems~\ref{thmU} and~\ref{thmU-half}.
Theorem~\ref{thmU} assumes that $f$ is suitably superlinear, it has a gradient structure ($f=\nabla F$)
and the problem $-\Delta U=\fzero(U)$ in $\R^n$ with $\fzero(U)=\lim_{\lambda\to0+}f(\lambda U)/\max_{|U|=\lambda}|f(U)|$
does not possess nontrivial nonnegative bounded solutions. It guarantees that
if problem \eqref{eq-fU} does not possess nontrivial nonnegative bounded stationary solutions,
then it does not possess nontrivial nonnegative bounded solutions. 
Theorem~\ref{thmU-half} is an analogue of Theorem~\ref{thmU}  in the case of a half-space.
The proofs of Theorems~\ref{thmU} and~\ref{thmU-half} are based on nontrivial modifications of arguments in \cite{Q21}.
Some of the assumptions in Theorem~\ref{thmU} can be weakened 
(and the proof can be significantly simplified) 
if we restrict the growth of the nonlinearity to $p\leq p_{sg}$, 
where
$$
p_{sg}:=\left\{
  \begin{aligned}
    &\frac{n}{n-2},&\quad &\text{if $n\ge 3$,}\\
    &\infty,&\quad &\text{if $n\in\{1,2\}$,}
  \end{aligned}\right.
$$
see Theorems~\ref{thmUK} and~\ref{thmUK2}. 
Finally, Theorem~\ref{thm-proportional} is 
a modification of the Liouville theorem \cite[Theorem~3]{Q16a}.
It assumes $N=2$ and $(f_1(u_1,u_2)-f_2(u_1,u_2))(u_1-u_2)\leq0$,
and it shows equality of components, hence reducing
the problem of nonexistence for system \eqref{eq-fU}
to the scalar case.

We also provide several scalar and vector-valued examples
which demonstrate the applicability of our results,
see Section~\ref{sec-ex}.
In particular, we study the range of parameters where
the Liouville theorems are true (see Example~\ref{benchmark})
and we also show that unlike in the case of scale-invariant problems,
the Liouville property need not have an open range in problems
without scale invariance (see Example~\ref{ex-pq}). 

\begin{remark} \rm
(i) Parabolic problems with non scale invariant nonlinearities have recently gained increasing interest. 
Various questions have been studied, such as local solvability for rough initial data, 
global existence/non-existence, or small diffusion limit; see, e.g., \cite{LRS13,FI18,LS20,FI22,FHIL}.
Concerning the blow-up behavior, beside the already mentioned works \cite{HZ22,S23}, we refer to \cite{Fuj,N15,DNZ,CS} 
(and see also the classical paper \cite{FML} for some early results).
Liouville theorems for such problems have also been studied, but they either require $n=1$ (see \cite{P20}),
or the growth of the nonlinearity $p$ is restricted by the condition $p\le p_F$ or $p<p_B$, where
$p_F:=(n+2)/n$ is the Fujita exponent and $p_B:=n(n+2)/(n-1)^2$ is the exponent of Bidaut-V\'eron \cite{BV98}.
The approach in \cite{BV98} requiring $p<p_B$ is based on a modification of the ``elliptic'' arguments in \cite{GS81}
(which only require $p<p_S$ in the elliptic case) and has been
applied to problems with non scale invariant nonlinearities in \cite{S23,JWY23}, for example.

(ii) Let us consider solutions $U$ of initial value parabolic problems
in the domain $D=\Omega\times(0,T)$ with $T\leq\infty$, and assume $U=0$ on $\partial\Omega\times(0,T)$ if $\Omega\ne\R^n$. 
Then for a large class of non scale invariant subcritical nonlinearities
a less precise a priori estimate of the form $|U(x,t)|\leq C(\|U(\cdot,0)\|_\infty,\eps)$ for $t<T-\eps$
can be obtained by energy, interpolation and bootstrap arguments (see \cite{Q03,I12}, for example).
This approach can be used even if the corresponding Liouville theorem and decay of solutions fail
(for example, if one considers sign-changing non-radial solutions of \eqref{eq-fU} in $\R^n\times(0,\infty)$
with $N=1$ and $f(U)=\tilde f(U)-U$, where $\tilde f$ satisfies the superlinearity condition \eqref{F2} and  
$c|U|^{q+1}\leq \tilde f(U)U\leq C|U|^{p+1}$
with $1<q\leq p<p_S$, $p-q$ sufficiently small, see \cite[Theorem~1.2]{Q03}).  
In addition, this approach can sometimes be used also for the rescaled equation \eqref{eq-w}, 
yielding optimal blow-up rate estimates for solutions of the original problem,
see \cite{GMS04,N15,HZ22,Z21}.
\end{remark}

\section{Liouville-type theorems}
\label{main}

In this section we present our main Liouville-type theorems. 
We have made the choice to state them under rather general assumptions, 
but we stress that they yield completely new, optimal results 
for many simple and concrete nonlinearities; see the examples in Section~\ref{sec-ex}.

By a nontrivial solution $U$ we understand
a classical solution $U=(u_1,\dots,u_N)\not\equiv(0,\dots,0)$.
By $|U(x,t)|$ (or $|\nabla U(x,t)|$, resp.) we denote the Euclidean norm
of $U(x,t)\in\R^N$ (or $\nabla U(x,t)\in\R^{nN}$, resp.);
the partial derivative $\frac{\partial U}{\partial t}$
is also denoted by $U_t$.
In the proofs we also use the notation $|A|$ to denote
the Lebesgue measure of a set $A\subset\R^n$,
and by $\#A$ we denote the cardinality of an arbitrary set $A$.

Set 
\begin{equation} \label{K}
\begin{aligned}
{\mathcal K} &:=\{U=(u_1,\dots,u_N)\in C^2(\R^n,\R^N):u_i\geq0,\ i=1,2,\dots,N\}, \\
 K &:= \{U=(u_1,\dots,u_N)\in\R^N:u_i\geq 0,\ i=1,2,\dots,N\}. 
\end{aligned}
\end{equation}
Given $f\in C(K,\R^N)$, we define  
$f^+:[0,\infty)\to[0,\infty):\lambda\mapsto\max\limits_{|U|=\lambda}|f(U)|$.

\begin{theorem} \label{thmU}
Let $1<p<p_S$. Assume that $f^+(\lambda)>0$ for $\lambda>0$,
\begin{equation} \label{FG}
f=\nabla F,\ \hbox{ where }\ 
F\in C^{1+\alpha}_{loc}(K,\R)\ \hbox{ for some }\ \alpha>0,\ F(0)=0,
\end{equation} 
\begin{equation} \label{F1}
\lim_{\lambda\to0+}\frac{f(\lambda U)}{f^+(\lambda)}= \fzero(U)
                    \quad\hbox{ locally uniformly for }\ U\in K,
\end{equation}
where $\fzero$ is homogeneous of degree $p$.
Assume also that for any $D>0$ there exist $\eta=\eta_D\in(0,p-1]$ such that
\begin{equation} \label{F2}
 \quad f(U)\cdot U\geq(2+\eta)F(U) \quad\hbox{ for all }\ U\in K,\ |U|\leq D.
\end{equation}
Finally assume that problems 
$-\Delta U=f(U)$ and $-\Delta U=\fzero(U)$ do not
possess nontrivial classical bounded solutions in ${\mathcal K}$. 
Then the system  
\begin{equation} \label{eq-U}
U_t-\Delta U=f(U)\quad\hbox{in}\quad \R^n\times\R
\end{equation}
does not possess any nontrivial 
classical bounded solution  satisfying $U(\cdot,t)\in{\mathcal K}$ for all $t\in\R$.
\end{theorem}

Set $\R^n_+:=\{x=(x_1,\dots,x_n)\in\R^n:x_n>0\}$,
$$
{\mathcal K}_+ :=\{U=(u_1,\dots,u_N)\in C^2(\overline{\R^n_+},\R^N):u_i\geq0\hbox{ in }\R^n_+,
\ u_i=0\hbox{ on }\partial\R^n_+,\  i=1,2,\dots,N\}.
$$
and consider the problem
\begin{equation} \label{eq-U-half}
\left.\begin{aligned}
 U_t-\Delta U &= \varphi(U) &\qquad&\hbox{in }\ \R^N_+\times\R, \\
 U &= 0 &\qquad&\hbox{on }\ \partial\R^N_+\times\R.
\end{aligned}\ \right\}
\end{equation}

\begin{theorem} \label{thmU-half}
Let all the assumptions of Theorem~\ref{thmU} be satisfied.
Assume, in addition, that problems \eqref{eq-U-half} with
$\varphi\in\{f,\fzero\}$ do not possess nontrivial classical bounded
stationary solutions in ${\mathcal K}_+$.
Then problem \eqref{eq-U-half} with $\varphi=f$
does not possess any nontrivial classical bounded solution
satisfying  
$U(\cdot,t)\in{\mathcal K}_+$ for all $t\in\R$.
\end{theorem}

\begin{theorem} \label{thmUK}
Let $1<p<p_{sg}$. Assume \eqref{FG}.
Assume also that for any $D>0$ there exist $C_D,c_D>0$ and $\xi=(\xi_1,\dots,\xi_N)\in(0,\infty)^N$
such that
\begin{equation} \label{F3}
|f(U)|\leq C_D|U|^q\ \hbox{ and }\ 
\xi\cdot f(U)\geq c_D|U|^p\quad \hbox{ for all }\ U\in K,\ |U|\leq D,
\end{equation} 
where $q\in[0,p]$, $p<1+\frac2n(q+1)$ if $q<p-1$.
Finally assume that the problem $-\Delta U=f(U)$ does not possess
nontrivial classical bounded solutions in ${\mathcal K}$. 
Then system \hbox{\eqref{eq-U}} does not possess any nontrivial 
classical bounded solution  satisfying $U(\cdot,t)\in{\mathcal K}$ for all $t\in\R$.
\end{theorem}

\begin{theorem} \label{thmUK2}
Let $1<p\leq p_{sg}$. Assume  
\eqref{FG}, \eqref{F2}, and let $F(U)>0$ for $U\neq0$.
Assume also that for any $D>0$ there exist $C_D,c_D>0$ and $\xi=(\xi_1,\dots,\xi_N)\in(0,\infty)^N$
such that
\begin{equation} \label{F4}
|f(U)|\leq C_D(|U|^p+G(U))\ \hbox{ and }\ 
\xi\cdot f(U)\geq c_D(|U|^p+G(U))\quad \hbox{ for all }\ U\in K,\ |U|\leq D,
\end{equation} 
where 
$G:K\to[0,\infty)$ is continuous and $q$-homogeneous, $q\in[0,p]$.
Finally assume that the problem $-\Delta U=f(U)$ does not possess
nontrivial classical bounded solutions in ${\mathcal K}$. 
Then system \hbox{\eqref{eq-U}} does not possess any nontrivial 
classical bounded solution  satisfying $U(\cdot,t)\in{\mathcal K}$ for all $t\in\R$.
\end{theorem}

\begin{remark} \label{rem1}\rm
(i) 
Given $I\subset\{1,2,\dots,N\}$ ($I$ may be empty), set
$${\mathcal K}_I:=\{U\in C^2: u_i\geq0\hbox{ for }i\in I\}$$
and $K_I:=\{U\in\R^N:u_i\geq0\hbox{ for }i\in I\}$
(hence ${\mathcal K_I}={\mathcal K}$ and $K_I=K$ if $I=\{1,2,\dots,N\}$).
Then Theorem~\ref{thmU} remains true
if we replace ${\mathcal K}$ and $K$ by ${\mathcal K}_I$ and $K_I$,
respectively.

If $n=1$ (or $n>1$ and we consider radial solutions only),  
$J\in\{0,1,2,\dots\}$, $\alpha_{ij}\in\R$ and $C_j\geq0$ ($i=1,2,\dots,N,\ j=1,2,\dots,J$), then
we also set
$${\mathcal K}_{I,z}={\mathcal K}_{I,z}(\alpha_{ij},C_j):=\Bigl\{U\in{\mathcal K_I^r}: 
  z\Bigl(\sum_{i=1}^N\alpha_{ij}u_i\Bigr)\leq C_j,\ j=1,2,\dots,J\Bigr\},$$
where ${\mathcal K}_I^r=\{U\in {\mathcal K}_I: U\hbox{ is radially symmetric}\}$ if $n>1$,
${\mathcal K}_I^r={\mathcal K}_I$ if $n=1$, and
$z(\varphi)$ is the number of sign changes of $\varphi$. 
Then we can replace ${\mathcal K}$ and $K$ in Theorem~\ref{thmU} by ${\mathcal K}_{I,z}$ and $K_I$,
respectively, but if $n>1$, then we also have to assume the elliptic Liouville theorems for $n=1$,
cf.~\cite[Corollary 2]{Q22}.
Notice that if $n>1$, $I=\{1,\dots,N\}$ and $J=0$, for example, then ${\mathcal K}_{I,z}$
consists of nonnegative radial functions. 

(ii) 
Assumption \eqref{FG} guarantees that $f$ is continuous, 
hence the function $\fzero$ in \eqref{F1} is continuous too.
Notice also that the nonexistence of nontrivial bounded solutions of 
$-\Delta U=f(U)$ and $-\Delta U=\fzero(U)$
in $\mathcal K$ implies $f(U)\ne0$ and $\fzero(U)\ne0$ for $U\in K\setminus\{0\}$.

On the other hand, if \eqref{F1} is satisfied with some (finite) function $\fzero$, 
then $\fzero$ is necessarily homogeneous (see \cite[Lemma~7.3]{QS24}). 
In the particular case when $f^+$ is asymptotically $p$-homogeneous, i.e.
\begin{equation} \label{fpp}
\lim\limits_{\lambda\to0+}\frac{f^+(\lambda)}{\lambda^p}=c>0,
\end{equation}
then \eqref{F1} is equivalent to
\begin{equation} \label{F1p}
\lim\limits_{\lambda\to0+}\frac{f(\lambda U)}{\lambda^p}=\hat f_0(U)\qquad
\hbox{locally uniformly for } \ U\in K,
\end{equation}
where $\hat f_0=c\fzero$. 
Finally, if $f$ itself is $p$-homogeneous, then \eqref{F1p} is true with $\hat f_0:=f$ 
(hence \eqref{F1} is true) 
and assumption \eqref{F2} is not needed in Theorem~\ref{thmU} due to \cite{Q22}.

(iii) The elliptic Liouville theorem \cite[Theorem~2.1]{QS24} 
requires 
\begin{equation} \label{F2QSa}
 2nF(U)-(n-2)f(U)\cdot U\geq c_D|U|^{p+1},\quad |U|\leq D,
\end{equation}
(see \cite[(2.3)]{QS24}).  
If $F(U)\geq \tilde c_D|U|^{p+1}$ for $|U|\leq D$ 
(which follows from \eqref{F3} if $N=1$), then
\eqref{F2QSa} can be written in the form 
\begin{equation} \label{F2QSb}
f(U)\cdot U\leq (p_S+1-\eta)F(U),\quad |U|\leq D,
\end{equation}
where $\eta=\eta_D>0$. 
The inequalities in \eqref{F2QSb} and 
\eqref{F2} are equivalent to
\begin{equation} \label{fF}
 (p_S+1-\eta)F(U)\geq f(U)U \geq (2+\eta)F(U).
\end{equation}
A sufficient condition for \eqref{fF} can be written in the form
\begin{equation} \label{fFsuff}
(p_S-\eta)f(U) \geq f'(U)U \geq (1+\eta) F(U). 
\end{equation}
\qed
\end{remark}

We now turn to the special case of systems of two equations 
and present results based on equality or proportionality of components.
The following proposition is a consequence of
the proof of \cite[Theorem~3]{Q16a}.

\begin{proposition} \label{propQ-JDE}
Let $b_1,b_2,c_1,c_2>0$, $c_1c_2>b_1b_2$,
$q\geq r>0$, $q+r>1$. 
Let $(u,v)$ be a nonnegative bounded solution of
the system
\begin{equation} \label{sysQ-JDE}
\left.\begin{aligned}
u_t-\Delta u &= u^r(c_1v^q-b_1u^q) \\
v_t-\Delta v &= v^r(c_2u^q-b_2v^q)
\end{aligned}\ \right\}\quad \hbox{in }\ X\times\R,
\end{equation}
where either $X=\R^n$ or $X=\R^n_+$.
If $X=\R^n_+$, then assume also $u=v=0$ on $\partial\R^n_+\times\R$.
Then there exists $K>0$ such that $u\equiv Kv$.
\end{proposition}

The assertion $u\equiv Kv$ in
Proposition~\ref{propQ-JDE} implies that $u$ is a bounded nonnegative
solution of the scalar equation
$u_t-\Delta u = cu^{r+q}$ in $X\times\R$ for some $c>0$
(and $u=0$ on $\partial\R^n_+\times\R$ if $X=\R^n_+$),
hence \cite[Theorem~1]{Q21} guarantees $u\equiv0$ provided $q+r<p_S$.

Next we will show that if $c_1=c_2$ and $b_1=b_2$,
then the power nonlinearities in Proposition~\ref{propQ-JDE} can be replaced
by more general nonlinearities.
Let $X\in\{\R^n,\R^n_+\}$.
We will consider the system 
\begin{equation} \label{sys-prop}
\left.\begin{aligned}
u_t-\Delta u &=\phi(u,v)k(u)[g(v)-\mu g(u)]\\
v_t-\Delta v &=\phi(u,v)k(v)[g(u)-\mu g(v)] 
\end{aligned}\ \right\}\quad \hbox{in }\ X\times\R
\end{equation}
complemented by the homogeneous Dirichlet boundary conditions if $X=\R^n_+$,
where
\begin{equation} \label{ass-prop1}
\hbox{$\phi,k,g$ are continuous, $\phi\geq c_M>0$ for $u,v\leq M$, $\mu\geq0$, $g(0)=0$, $g>0$ on $(0,\infty)$}
\end{equation}
($\phi$ could also depend on $x,t,\nabla u,\nabla v,\dots$).
Set $\varphi:=kg$ if $\mu>0$, $\varphi:=g$ if $\mu=0$, and assume that
\begin{equation} \label{ass-prop2}
 \hbox{$\varphi$ is $C^1$ on $(0,\infty)$ and $\varphi'(t)>0$ for $t>0$.}
\end{equation}
Fix also $\eps=0$ if $\mu>0$, $\eps>0$ if $\mu=0$, set $\tilde k:=k-\eps$
and assume that
\begin{equation} \label{ass-propk}
\hbox{$\tilde k(0)\geq0$ and $\tilde k/g$ is nonincreasing on $(0,\infty)$.}
\end{equation}

\begin{theorem} \label{thm-proportional}
Assume \eqref{ass-prop1}--\eqref{ass-propk}. 
Let $(u,v)$ 
be a nonnegative bounded 
strong\footnote{The couple $(u,v)$ is a strong solution if 
$(u,v)\in W^{2,1;r}_{loc}$ for all $r\in(1,\infty)$ and the equation is satisfied a.e. 
Moreover, when boundary conditions are imposed, a strong solution is assumed to be continuous up the boundary.
In the present paper we mainly deal with classical solutions;
the more general class of strong solutions in
Theorem~\ref{thm-proportional} is required by the application of this
theorem in \cite{QS24}.}
solution of \eqref{sys-prop}
satisfying the homogeneous Dirichlet boundary conditions if $X=\R^n_+$. Then $u\equiv v$.
\end{theorem}

\section{Universal estimates}
\label{sec-ub}

Let $\fzero$ and $\finfty$ denote the limiting functions defined by \eqref{limg} and \eqref{limh} below, respectively
(if those limits exist).
Similarly as in the elliptic case (see \cite[Section~4]{QS24}), 
Liouville theorems for problems $U_t-\Delta U=\varphi(U)$ with $\varphi\in\{f,\fzero,\finfty\}$
guarantee various universal estimates for solutions of the problem
$U_t-\Delta U=f(U)$ in an arbitrary nonempty open set $D$ in $\R^n\times\R$.
These estimates can be formulated in terms of the parabolic distance
$\hbox{dist}(X,\partial D):=\inf\{d(X,Y):Y\in\partial D\}$, where
$d(X,Y):=|x-y|+|t-s|^{1/2}$, $X=(x,t)$ and $Y=(y,s)$.

Let us start with a simple proposition, whose first part is a parabolic analogue
of \cite[Proposition~4.1]{QS24} 
and whose proof is a straightforward modification of the proof 
of \cite[Proposition~4.1]{QS24}.

\begin{proposition}  \label{propdecayPar}
Let $f\in C^\alpha_{loc}(K,\R^N)$ for some $\alpha>0$, and let $\Lambda>0$.

(i) Assume that the system $V_t-\Delta V=f(V)$
admits no nontrivial classical solution $V:\R^n\times\R\to K$ with $\|V\|_\infty\le \Lambda$.
Then 
$\lim_{R\to\infty} \eta(R)=0,$
where\footnote{with the convention $\sup(\emptyset)=0$}
$$
\begin{aligned}
\eta(R)&=\hbox{$\sup\bigl\{|U(x,t)|: D\subset\R^n\times\R$ is open, $U:D\to K$ is a bounded classical solution}\\
&\qquad\quad\hbox{of $U_t-\Delta U=f(U)$, $\|U\|_\infty\le \Lambda$
and $(x,t)\in D$ with ${\rm dist}((x,t),\partial D)\ge R\bigr\}$}.
\end{aligned}$$

(ii) Assume that the system $V_t-\Delta V=f(V)$
admits no nontrivial classical solution $V:\R^n\times\R\to K$ with $\|V\|_\infty\le \Lambda$
and no nontrivial classical solution $V:\R^n_+\times\R\to K$ with $\|V\|_\infty\le \Lambda$
and $V=0$ on $\partial\R^n\times\R$. Let $D=\R^n_+\times(0,\infty)$.
Then $\lim_{R\to\infty} \tilde\eta(R)=0,$
where
$$
\begin{aligned}
\tilde\eta(R)&=\hbox{$\sup\bigl\{|U(x,t)|: t\ge R$ and 
$U:D\to K$ is a bounded classical solution of} \\ 
&\qquad\quad\hbox{$U_t-\Delta U=f(U)$ 
satisfying $U=0$ on $\partial\R^n_+\times(0,\infty)$, $\|U\|_\infty\le \Lambda\bigr\}$}.
\end{aligned}$$
 
\end{proposition}

\begin{remark} \label{rem-decay} \rm
If $D=\{(x,t)\in \R^n\times(0,\infty): 0<x_n<1\}$ and $U=0$ on $\{(x,t)\in\partial D:t>0\}$,
for example, then the decay estimate in Proposition~\ref{propdecayPar}(ii) fails in general: 
Consider $N=n=1$, $f(u)=u^p$, $p>1$,
and the positive solution of $-u''=u^p$ in $(0,1)$ satisfying $u(0)=u(1)=0$.  
\end{remark}

Next we proceed with an analogue of \cite[Theorem 4.3]{QS24}. 

\begin{theorem} \label{thmUB}
Let $f\in C^\alpha_{loc}(K,\R^N)$ for some $\alpha>0$, 
$\Lambda>0$,
and let $U:D\to K$ be a classical solution of \eqref{eq-U} in $D\subset\R^n\times\R$.  

(i) Assume $f^+(\lambda)>0$ for $\lambda\geq\Lambda$ and
\begin{equation} \label{limh} 
  \lim_{\lambda\to\infty}\frac{f(\lambda U)}{f^+(\lambda)}= \finfty(U)
  \quad\hbox{ locally uniformly for }\ U\in K,
\end{equation}
where $\finfty$ is homogeneous of order $p>1$.
Assume also that the problem $V_t-\Delta V=\finfty(V)$
does not possess nontrivial classical bounded entire solutions $V:\R^n\times\R\to K$.
Set $D_\Lambda:=\{X\in D:|U(X)|\geq\Lambda\}$.
Then there exists a constant $C=C(n,f,\Lambda)$ (independent of $D$ and $U$) 
such that 
\begin{equation} \label{UBh}
 |f(U(X))|\leq C|U(X)|(1+\hbox{\rm dist}^{-2}(X,\partial D))
\quad\hbox{for all } X\in D_\Lambda.
\end{equation}

(ii) Assume $f^+(\lambda)>0$ for $\lambda\in(0,\Lambda]$ and
\begin{equation} \label{limg} 
  \lim_{\lambda\to0+}\frac{f(\lambda U)}{f^+(\lambda)}= \fzero(U)
  \quad\hbox{ locally uniformly for }\ U\in K,
\end{equation}
where $\fzero$ is homogeneous of order $q>1$.
Assume also that problems $V_t-\Delta V=\varphi(V)$
do not possess nontrivial classical bounded entire solutions $V:\R^n\times\R\to K$
for $\varphi\in\{f,\fzero\}$. Let $\|U\|_\infty\leq\Lambda$.
Then there exists a constant $C=C(n,f,\Lambda)$ (independent of $D$ and $U$) 
such that
\begin{equation} \label{UB}
 |f(U(X))|\leq C|U(X)|\hbox{\rm dist}^{-2}(X,\partial D)
\quad\hbox{for all } X\in D.
\end{equation}

(iii)
Assume  $f^+(\lambda)>0$ for $\lambda>0$.
Let \eqref{limh} and \eqref{limg} be true,
where $\finfty$ and $\fzero$ is homogeneous of order $p>1$ and $q>1$, respectively.
Assume also that problems $V_t-\Delta V=\varphi(V)$ do not possess 
nontrivial classical bounded entire solutions $V:\R^n\times\R\to K$
for $\varphi\in\{f,\fzero,\finfty\}$.
Then there exists a constant $C=C(n,f)$ (independent of $D$ and $U$) such that 
\eqref{UB} is true.
\end{theorem}

\begin{remark}\rm \label{rem-Liouv}
(i) If \eqref{FG} and \eqref{F2} are true,
then instead of assuming the nonexistence of nontrivial nonnegative bounded entire solutions
of \eqref{eq-U} in Theorem~\ref{thmUB} it is sufficient to assume 
the nonexistence of nontrivial nonnegative bounded stationary solutions of \eqref{eq-U}.
This follows from Theorem~\ref{thmU}.

(ii) Assume \eqref{FG}, \eqref{F2}, and let $p,q\in(1,p_S)$.
Assume \eqref{limg}, \eqref{limh}, where $\fzero$ and $\finfty$
are homogeneous of degree $p$ and $q$, respectively,
$\fzero,\finfty\in C^\alpha_{loc}$.
If systems $U_t-\Delta U=\varphi(U)$ do not posssess
nontrivial nonnegative bounded stationary solutions for $\varphi\in\{f,\fzero,\finfty\}$,
then the conclusion of Theorem~\ref{thmUB}(iii) remains true.
In fact,  the nonexistence of nontrivial entire solutions
of $U_t-\Delta U=\varphi(U)$ follows from  Theorem~\ref{thmU} (if $\varphi=f$) 
and \cite{Q22} (if $\varphi\in\{\fzero,\finfty\}$).

(iii) Let $\Omega\subset\R^n$ have uniformly smooth boundary and $T>0$.
Consider solutions $U$ of \eqref{eq-U} in $D=\Omega\times(0,T)$,
satisfying also
the Dirichlet boundary condition $U=0$ on $\partial\Omega\times(0,T)$.
Then similarly as in \cite[Section~4]{PQS07b}, by using Liouville
theorems both in the whole space $\R^n\times\R$ and the half-space
$\R^n_+\times\R$ in the proof of Theorem~\ref{thmUB}(i),
one can obtain universal estimates of the form
$$  |f(U(X))|\leq C|U(X)|(1+t^{-1}+(T-t)^{-1}) $$ 
instead of \eqref{UBh}.
\qed
\end{remark}

We next consider the following Lane-Emden type system
with $N=2$ and $U=(u,v)$:
\begin{equation} \label{LEsys1}
\left.
  \begin{aligned}
 u_t-\Delta u&=f_1(v), \\ 
 v_t-\Delta v&=f_2(u),   
   \end{aligned}
   \quad\right\}
\end{equation}
where
\begin{equation} \label{hypLE1}
\hbox{$f_i\in C([0,\infty))$, $f_i=f_i(s)$ are $C^1$ and positive for $s>0$ large,}
\end{equation}
\begin{equation} \label{hypLE2}
p,q>0,\quad
\hbox{$L_1(s):=s^{-p} f_1(s)$, $L_2(s):=s^{-q} f_2(s)$ satisfy }
\lim_{s\to\infty} \frac{sL'_i(s)}{L_i(s)}=0.
\end{equation}
System \eqref{LEsys1} is a modification of the classical Lane-Emden system
\begin{equation} \label{LEsys2}
\left.
  \begin{aligned}
 u_t-\Delta u&=v^{p}, \\ 
 v_t-\Delta v&=u^{q}.   
   \end{aligned}
   \quad\right\}
\end{equation}
If $p_0\in(1,p_S)$, then the arguments in \cite{FIM17} combined with \cite[Theorem~1]{Q21}
guarantee the existence of $\eps>0$ such that  
\begin{equation} \label{LE-Liouv}
\hbox{system \eqref{LEsys2} does not possess positive classical solutions in
$\R^n\times\R$} 
\end{equation}
provided $|p-p_0|+|q-p_0|\leq\eps$ (cf.~also \cite[Remark 32.8a]{QS19}).
Assertion \eqref{LE-Liouv} is also true if 
$pq>1$ and $\max(\alpha,\beta)\geq n$, where
$\alpha=\frac{2(p+1)}{pq-1}$ and $\beta=\frac{2(q+1)}{pq-1}$,
see \cite[Theorem~32.7]{QS19}. 

A straightforward modification of the proof of \cite[Theorem~4.3]{PQS07a}
shows that property \eqref{LE-Liouv} guarantees universal estimates
\begin{equation} \label{estimLE2}
 u(X)\leq C(1+\hbox{\rm dist}^{-\alpha}(X,\partial D)), \qquad
   v(X)\leq C(1+\hbox{\rm dist}^{-\beta}(X,\partial D)) 
\end{equation}
for positive solutions of \eqref{LEsys1} in $D:=\Omega\times(0,T)$,
provided $f_1,f_2$ satisfy \eqref{hypLE1}, $pq>1$ and 
the functions $L_1,L_2$ in \eqref{hypLE2} satisfy
\begin{equation} \label{f1f2pq}
\lim_{s\to\infty}L_1(s)=\lim_{s\to\infty}L_2(s)=1.
\end{equation}
Notice that if $p\ne q$, then one cannot use Theorem~\ref{thmUB}(i) to obtain \eqref{estimLE2} since
the limit $\finfty$ in \eqref{limh} satisfies $\finfty(U)=(v^p,0)$ or $\finfty(U)=(0,u^q)$,
hence the system $U_t-\Delta U=\finfty(U)$ possesses nontrivial entire constant solutions.
However, one can 
obtain a generalization of estimates \eqref{estimLE2}
for any $f_1,f_2$ satisfying \eqref{hypLE1}-\eqref{hypLE2}
(without assuming \eqref{f1f2pq}).
In fact, set
$$h_i(s)=sf_i(s),\quad s\ge 0.$$
One can check (see \cite{QS24}) that there exists $s_0>0$ such that the functions
\begin{equation} \label{s0}
\begin{aligned}
&\hbox{$h_1$, $h_2$, $h_1^{-1}$, $h_2^{-1}$, $h_1^{-1}\circ h_2$ and $h_2^{-1}\circ h_1$}\\
&\hbox{are defined, positive and increasing on $[s_0,\infty)$.}
   \end{aligned}
\end{equation}
The proof of the following theorem follows by a straightforward 
modification of \cite[Theorem~4.1]{QS24}.  

\begin{theorem} \label{thmLEpert1}
Assume \eqref{hypLE1}-\eqref{hypLE2}, $pq>1$ and \eqref{LE-Liouv}.
Then there exists a constant $C=C(n,f_1,f_2)>0$ such that, if $\Omega$ is an arbitrary domain of
$\R^n$, $T>0$
and $(u,v)$ is a nonnegative classical solution of~\eqref{LEsys1} in
$D:=\Omega\times(0,T)$, then we have the estimates
\begin{equation} \label{estimLE1}
  \begin{aligned}
\frac{f_2(u)}{(h_1^{-1}\circ h_2)(u)} &\le C(1+\hbox{\rm dist}^{-2}(X,\partial D)),\quad X\in D_1,\\
\frac{f_1(v)}{(h_2^{-1}\circ h_1)(v)} &\le C(1+\hbox{\rm dist}^{-2}(X,\partial D)),\quad X\in D_2,
  \end{aligned}
\end{equation}
where $D_1=\{X\in D:u(X)\ge s_0\}$, $D_2=\{X\in D:v(X)\ge s_0\}$
and $s_0$ satisfies \eqref{s0}.
\end{theorem}

\section{Examples} 
\label{sec-ex}
\vskip2mm

\begin{example} \label{benchmark} \rm
Let $N=1$, $n\geq3$, $n/(n-2)<p<p_S$, $a\geq0$, $U\geq0$, and 
$$f(U)=(1+a\min(1,U^{p-1}))U^p.$$
If $a=a_0:=2p/[(n-2)p-n]$, then the equation
$-\Delta U=f(U)$ in $\R^n$ possesses an explicit
positive solution, while if $a<\theta a_0$, where
$$\theta:=\frac{n+2-(n-2)p}{4p}\,\frac{2(n-2)p-n}{(n-2)p-2}<1,$$
then that equation does not possess nontrivial
nonnegative solutions, see \cite[Subsection~3.1]{QS24}.  
If $a<\theta a_0$, then Theorems~\ref{thmU} and~\ref{thmUB}
guarantee the universal estimates \eqref{UB} for nonnegative
solutions of the equation $U_t-\Delta U=f(U)$, while
these estimates fail if $a=a_0$.
Notice also that Theorem~\ref{thmUB}(i) 
guarantees the
singularity estimates \eqref{UBh} for all $a\geq0$.
\end{example}

\begin{example} \label{ex-HZ} \rm
Let $N=1$, $1<p<p_S$, $a\in\R$, $f(U)=|U|^{p-1}U\log^a(2+U^2)$. 
It has recently been proved in \cite{HZ22}
that the (possibly sign-changing) solutions to the 
problem $U_t-\Delta U=f(U)$ in $\R^n\times(0,T)$ satisfy 
the blow-up rate estimate
\begin{equation} \label{BU-HZ}
\|U(t)\|_\infty\leq C(T-t)^{-1/(p-1)}(-\log(T-t))^{-a/(p-1)}
\quad\hbox{for } t\in(T-\eps,T),
\end{equation}
where $C$ depends on $U$ (see \cite[Theorem~2]{HZ22}).
The proof is based on a highly nontrivial 
adaptation of the method from \cite{GMS04},
where the scale-invariant case $a=0$ was considered. 
In the case of positive solutions,
the blow-up rate estimate \eqref{BU-HZ}
can also be obtained as a simple consequence of Theorem~\ref{thmUB}(i)
or \cite[Theorem~3.1]{S23}, 
cf.~\cite[Remark~2(i)]{S23}.
In fact, 
Theorem~\ref{thmUB}(i) implies estimate \eqref{UBh}
which guarantees the blow-up rate estimate \eqref{BU-HZ}.
In addition, estimate \eqref{UBh} is universal
(it does not depend on $U$) and it can also be easily established
for a whole class of nonlinearities.

Assume $U\geq0$.
We will show that if $p<p_{sg}$ or
\begin{equation} \label{HZa}
a\in(-(p-1)c_\ell,(p_S-p)c_\ell],\quad\hbox{where }\ 
c_\ell:=\inf_{U>0}\log(2+U^2)\frac{2+U^2}{2U^2}>1,
\end{equation}
then the assumptions of Theorem~\ref{thmUB}(iii) are satisfied,
hence we obtain the universal estimates \eqref{UB} which
include both estimates of singularities and decay of solutions.
Example~\ref{benchmark} shows that even if the singularity
estimates \eqref{UBh} are true for any $a\in\R$ and $p<p_S$,
one cannot expect such general result in the case
of estimates \eqref{UB}.

Hence assume that $p<p_{sg}$ or \eqref{HZa} is true.
\cite[Lemma~8.1]{QS24} 
implies that \eqref{limg},\eqref{limh}
are true with $\fzero(U)=\finfty(U)=U^p$.
We will show that the required Liouville theorems for 
$\varphi\in\{f,\fzero,\finfty\}$ in Theorem~\ref{thmUB} are true.
If $\varphi\in\{\fzero,\finfty\}$, then this follows from \cite{Q21}.
If $\varphi=f$, then we can use  
Theorem~\ref{thmUK} and \cite[Theorem 8.4]{QS19} if $p<p_{sg}$
and Theorem~\ref{thmU} if $p\geq p_{sg}$:
The inequality $a>-(p-1)c_\ell$ implies
$(f(U)U)'\geq(2+\eta)f(U)$ for some
$\eta\in(0,p-1]$, hence \eqref{F2} is true.
The nonexistence of positive solutions of 
$\Delta U+f(U)=0$ in $\R^n$, $n>2$, follows from
the inequality $a\leq(p_S-p)c_\ell$ and \cite[Theorem 1.3]{LZ03}.
Notice also that the condition
$a\in(-(p-1)c_\ell,(p_S-p)c_\ell)$ is equivalent to \eqref{fFsuff}.
\end{example}

\begin{example} \label{ex-pq} \rm
Let $N=1$. If $f(U)=U^p+U^q$ with $1<p<q\leq p_S$,
then \cite{LZ03} and Theorem~\ref{thmU}
guarantee Liouville theorems for bounded solutions of $U_t-\Delta U=\varphi(U)$
with $\varphi\in\{f,\fzero\}$, $\fzero(U):=U^p$,
hence Theorem~\ref{thmUB}(ii) implies
the decay estimate $|U(x,t)|\leq Ct^{-1/(p-1)}$ for global bounded
nonnegative solutions of $U_t-\Delta U=f(U)$
in $\R^n\times(0,\infty)$.
In addition, if $q<p_S$, then the Liouville theorem
is true also if $\varphi(U)=\finfty(U):=U^q$,
hence Theorem~\ref{thmUB} guarantees the
universal estimates \eqref{UB} for nonnegative solutions of $U_t-\Delta U=f(U)$ in $D$.

Notice that if $p_k:=p_{sg}+\frac1k$, $q_k:=2p_k-1=p_S+\frac2k$ and
$f_k(U)=U^{p_k}+U^{q_k}$, where $k\in{\mathbb N}$ is large, 
then the equations $-\Delta U=f_k(U)$ in $\R^n$, 
hence also the equations $U_t-\Delta U=f_k(U)$ in $\R^n\times\R$
possess explicit bounded positive solutions (see \cite[Subsection~3.2]{QS24}),
but the limiting equation
$U_t-\Delta U=U^{p_{sg}}+U^{p_S}$ 
does not possess bounded positive entire solutions. 
Similarly, if $p\in(p_{sg},p_S)$ is fixed, then the equation $-\Delta U=U^p+U^q$
possesses bounded positive entire solutions for $q\in(p_S,p_S+\eps)$
(see \cite{QS24} again),
but the equation $U_t-\Delta U=U^p+U^{p_S}$
does not possess bounded positive entire solutions.
These examples show that unlike in the case of scale-invariant problems
(see \cite[Proposition 21.2b and subsequent comments]{QS19}),
the Liouville property need not have an open range in problems
without scale invariance. 
On the other hand, assume \eqref{limh} and \eqref{limg},
and let the Liouville theorem be true for the triplet $(f,\fzero,\finfty)$
(i.e.~for all problems
$U_t-\Delta U=\varphi(U)$ with $\varphi\in\{f,\fzero,\finfty\}$).
If $(f_k,(\bar{f_k})_0,(\bar{f_k})_\infty)\to(f,\fzero,\finfty)$ in a suitable way, 
then similar contradiction arguments and in the proof of Theorem~\ref{thmUB} show that
the Liouville theorem is true for $(f_k,(\bar{f_k})_0,(\bar{f_k})_\infty)$ with $k$ large enough.
\end{example}

\begin{example} \label{cubic-quintic} \rm
Let $n=3$, $N=2$, $U=(u,v)$,  
$G(U)=\frac14(u^4+v^4+2\alpha u^2v^2)$,
$H(U)=\frac b6(u^6+v^6+2\beta u^3v^3)$ (or $H(U)=bu^3v^3$),
$F(U)=G(U)+H(U)$, 
where $\alpha>-1$, $b>0$, $\beta\geq-1$.
Let $f=\nabla F$, hence $\fzero=\nabla G$.
Then
\cite[Corollary~2.1]{QS24} 
implies that the elliptic problems
$-\Delta U=\varphi(U)$ do not possess nontrivial nonnegative bounded solutions in $\R^n$
for $\varphi\in\{f,\fzero\}$ 
and Theorem~\ref{thmU} implies that the same remains true
for the parabolic problems $U_t-\Delta U=\varphi(U)$ in $\R^N\times\R$.
Consequently, if $U$ is a bounded nonnegative solution of $U_t-\Delta U=f(U)$
in $\R^n\times(0,\infty)$, then Theorem~\ref{thmUB}(ii) 
implies 
$|U(x,t)|\leq C/\sqrt{t}$. 
\end{example}

\begin{example} \rm
Let $N=2$, $n\leq4$, $f=\nabla F$, where $F(U)=u_1^2u_2+\frac13 u_2^3
+\sum_{(i,j)\in E} a_{ij}u_1^iu_2^j$,
$a_{ij}\geq0$, 
$E=\{(i,j): i,j\ge 0, 4\le i+j\le M\}$,
$M\geq4$.
If $-\Delta U=f(U)$, then
$w:=u_1+u_2$ satisfies $-\Delta w\geq w^2$, hence $w=0$ (see \cite{QS19}).
If $U_t-\Delta U=f(U)$, then $w:=u_1+u_2$ satisfies
$w_t-\Delta w\geq w^2$, but this inequality
does not have nontrivial solution only if $n\leq2$ (see \cite[Proposition 21.14]{QS19}).
However, one can use  Theorem~\ref{thmU} to prove $U=0$ whenever $n\le4$.
Since the required Liouville theorem is also true in the case of the function $\fzero$,
Theorem~\ref{thmUB}(ii) implies the
universal estimates \eqref{UB} for bounded solutions. 
\end{example}

\begin{example} \rm
Let $N=2$, $n=3$, $f(U)=(2u_1^3\log(3+u_1^2)-u_1u_2^2,2u_2^3\log(3+u_2^2)-u_1^2u_2)$.
Then $f^+(\lambda)=2\lambda^3\log(3+\lambda^2)$,
$\fzero(U)=(u_1^3-u_1u_2^2/2\log3,u_2^3-u_1^2u_2/2\log3)$, $\finfty(U)=(u_1^3,u_2^3)$,
and \eqref{fF} is true (for $f$ and $\fzero$).
Consequently, one can use 
\cite[Theorem 2.1]{QS24} 
together with Theorems~\ref{thmU} and~\ref{thmUB}(iii) 
in order to show that the universal estimate \eqref{UB} is true
for any nonnegative solution of $U_t-\Delta U=f(U)$ in $D$. 
\end{example}
 
\begin{example} \rm
Assume $a\in\R$, $b_1,b_2\geq0$, $b_1+b_2>0$, $p>1$, $\mu>0$, $X=\R^n$, and
let $U=(u,v)$ be a bounded nonnegative solution of
system \eqref{sys-prop} with $\phi\equiv1$, $k(u)=b_1+b_2u$ and $g(u)=u^p\log^a(2+u^2)$. 
Then there exists $a_0=a_0(p,b_1,b_2)>0$ such that 
if $a>-a_0$, then the assumptions of Theorem~\ref{thm-proportional}
are satisfied, hence $u\equiv v$. If $\mu>1$, then this fact and Lemma~\ref{lemQ-JDE} imply $u\equiv v\equiv0$.
If $\mu<1$, then assume $n>2$ and $p<p_S-1$ (or $p<p_S$ and $b_2=0$).
Then there exists $a_1=a_1(p,b_1,b_2)>0$ such that
$k(u)g(u)u^{-p_S}$ is nonincreasing if $a<a_1$, hence $a\in(-a_0,a_1)$ implies
$u\equiv v\equiv0$ due to \cite{LZ03} and Theorem~\ref{thmU}.
\end{example}

\begin{example} \rm
Consider the system
$$ 
\begin{aligned}
u_t-\Delta u &= k(u)v^p,\\
v_t-\Delta v &= k(v)u^p,
\end{aligned}
$$
where $n>2$, $p>1$ and $k(u)=1+bu^r$ with $b>0$ and $r\in(0,p]$, $r+p\leq p_S$. 
Theorem~\ref{thm-proportional} implies that
any bounded nonnegative solution $U=(u,v)$ in $\R^n\times\R$ satisfies $u\equiv v$,
and \cite{LZ03} with Theorem~\ref{thmU}
imply $u\equiv v\equiv0$.
Since the required Liouville theorem is also true in the case of the function
$\fzero(U)=(v^p,u^p)$, Theorem~\ref{thmUB}(ii) implies the
universal estimates \eqref{UB} for bounded solutions. 
\end{example}

\section{Proof of Theorem~\ref{thmU}} 
\label{sec-proof}

First notice that \cite[Lemma~7.3]{QS24} 
implies
\begin{equation} \label{fp0}
(\forall\theta>0)\,(\exists\lambda_\theta>0)\quad
\lambda\in(0,\lambda_\theta)\ \Rightarrow\ 
 \lambda^{p+\theta}\leq f^+(\lambda)\leq\lambda^{p-\theta}.
\end{equation}
 
Assume to the contrary that there exists
a nontrivial bounded solution $U$ of \eqref{eq-U} satisfying $U(\cdot,t)\in{\mathcal K}$ for all $t$.
Then there exists $D>0$ such that
\begin{equation} \label{bound-u}
 |U(x,t)|+|\nabla U(x,t)|\leq D \qquad\hbox{ for all }x\in\R^n,\ t\in\R.
\end{equation}

We may also assume $U(0,0)\ne0$.

Set $\beta:=1/(p-1)$.
Let $D$ and $\eta$ be as in \eqref{bound-u} and \eqref{F2}, respectively. 
By $C,C_0,C_1,\dots,$\allowbreak$c,c_0,c_1,\dots$ we will denote
positive constants which depend only on $n,p,D,\eta$ and $f$;
the constants $C,c$ may vary from step to step.
Similarly by $\eps$ and $\delta$ we will denote
small positive constants which may vary from step to step
and which only depend on $n,p,D,\eta$ and $f$.
Finally, $M=M(n,p)$ will denote a positive integer
(the number of bootstrap steps). 
The proof will be divided into several steps.

\goodbreak
{\bf Step 1: Initial estimates.}
For $a\in\R^n$ and $k=1,2,\dots$ we set
$$W(y,s)=W_k^a(y,s):=(k-t)^\beta U(y\sqrt{k-t}+a,t),\qquad\hbox{where }\  s=-\log(k-t),\ \ t<k.$$
Set also $s_k:=-\log k$ and notice that
$W=W_k^a$ solves the problem
\begin{equation} \label{eq-w}
\left.\begin{aligned}
 W_s &=\Delta W-\frac12 y\cdot\nabla W-\beta W+\phi^{-p}f(\phi W) \\
     &=\frac1\rho\nabla\cdot(\rho\nabla W)-\beta W+\phi^{-p}f(\phi W)\qquad \hbox{in }
\R^n\times\R,
\end{aligned}\quad\right\}
\end{equation}
where $\phi=\phi(s):=e^{\beta s}$ and $\rho=\rho(y):=e^{-|y|^2/4}$. In addition,
$$  W_k^a(0,s_k)=k^\beta U(a,0),$$  
and
\begin{equation} \label{bound-w2}
\left.
\begin{aligned}
 \|W_k^a(\cdot,s)\|_\infty &\leq   C_0k^\beta  \\
  \|\nabla W_k^a(\cdot,s)\|_\infty &\leq   C_0k^{\beta+1/2} 
\end{aligned}
\ \right\}
\quad \hbox{for }\
s\in[s_k-M-1,\infty), 
\end{equation}
where $C_0:=De^{(M+1)(\beta+1/2)}$. 

Set
$$
E_k^a(s) :=\frac12\int_{\R^n}\bigl(|\nabla W_k^a|^2+\beta|W_k^a|^2\bigr)(y,s)\rho(y)\,dy
 -\phi^{-(p+1)}(s)\int_{\R^n}F(\phi W_k^a)(y,s)\rho(y)\,dy.
$$
Then, supressing the dependence on $k,a$ in our notation, we obtain
\begin{equation} \label{GK1}
\begin{aligned}
\frac12\frac{d}{ds}\int_{\R^n}&|W|^2(y,s)\rho(y)\,dy  
=-2E(s)\\
&+\phi^{-(p+1)}(s)\int_{\R^n}[f(\phi W)\cdot\phi W-2F(\phi W)](y,s)\rho(y)\,dy 
\end{aligned}
\end{equation}
and
\begin{equation} \label{GK2}
\begin{aligned}
E_s(s)=
\frac{d}{ds}E(s)&=-\int_{\R^n}|W_s|^2(y,s)\rho(y)\,dy \\
 &+\beta\phi^{-(p+1)}(s)\int_{\R^n}[(p+1)F(\phi W)-f(\phi W)\cdot\phi W](y,s)\rho(y)\,dy. 
\end{aligned}
\end{equation}
Consequently, setting 
$${\mathcal E}(s)={\mathcal E}^a_k(s):=E(s)-\frac{C_E}2\int_{\R^n}|W|^2(y,s)\rho(y)\,dy, \quad
\hbox{ where }\ C_E:=\frac{p-1-\eta}{\eta(p-1)}, $$
and using \eqref{F2} we obtain
$$ {\mathcal E}_s(s)\leq -\int_{\R^n}|W_s|^2(y,s)\rho(y)\,dy+2C_EE(s), $$
hence 
$$ {\mathcal E}_s(s)-2C_E{\mathcal E}\leq -\int_{\R^n}|W_s|^2(y,s)\rho(y)\,dy
     +C_E^2\int_{\R^n}|W|^2(y,s)\rho(y)\,dy, $$
which implies
\begin{equation} \label{dEs}
\begin{aligned}
 (e^{2C_E(s_k-s)}{\mathcal E}(s))_s \leq&
 -e^{2C_E(s_k-s)}\int_{\R^n}|W_s|^2(y,s)\rho(y)\,dy \\
 &+C_E^2e^{2C_E(s_k-s)}\int_{\R^n}|W|^2(y,s)\rho(y)\,dy.
\end{aligned}
\end{equation}
Assuming $\sigma_1,\sigma_2\in[s_k-M-1,s_k+M+1]$,  
\eqref{dEs} guarantees the existence of $C_1=C_1(n,p,M,D,\eta)\geq1$ such that
\begin{equation} \label{estWs}
\int_{\sigma_1}^{\sigma_2}\int_{\R^n}|W_s|^2(y,s)\rho(y)\,dy\,ds
\leq C_1\Bigl(|{\mathcal E}(\sigma_2)|+|{\mathcal E}(\sigma_1)|
+\int_{\sigma_1}^{\sigma_2}\int_{\R^n}|W|^2(y,s)\rho(y)\,dy\,ds\Bigr).
\end{equation}
Set 
\begin{equation}\label{c0}
c_0:=\frac1{4C_1}\leq\frac12,\qquad d:=\frac12 c_0,
\end{equation}
$$\varphi(\sigma):=\int_{\sigma_1}^\sigma\int_{\R^n}|W|^2(y,s)\rho(y)\,dy,\quad
\psi:=\int_{\sigma_1}^{\sigma_2}\int_{\R^n}|W_s|^2(y,s)\rho(y)\,dy,$$
and assume $\sigma_2-\sigma_1\leq c_0$.
If $\sigma\in[\sigma_1,\sigma_2]$, then
$$ \begin{aligned}
  0\leq\varphi'(\sigma) &=\int_{\R^n}|W|^2(y,\sigma)\rho(y)\,dy =\varphi'(\sigma_1)
  +2\int_{\sigma_1}^\sigma\int_{\R^n}W\cdot W_s(y,s)\rho(y)\,dy \\
  &\leq \varphi'(\sigma_1)+2\sqrt{\psi}\sqrt{\varphi(\sigma_2)}
  \leq\varphi'(\sigma_1)+\psi+\varphi(\sigma_2),
\end{aligned}$$
hence 
$$ \varphi(\sigma_2)\leq c_0(\varphi'(\sigma_1)+\psi+\varphi(\sigma_2)). $$
This estimate and \eqref{c0} yield
$$ \begin{aligned}
\int_{\sigma_1}^{\sigma_2}\int_{\R^n}&|W|^2(y,s)\rho(y)\,dy\,ds =
 \varphi(\sigma_2)\leq \frac{c_0}{1-c_0}(\varphi'(\sigma_1)+\psi) \\
 &\leq \int_{\R^n}|W|^2(y,\sigma_1)\rho(y)\,dy
 +\frac1{2C_1}\int_{\sigma_1}^{\sigma_2}\int_{\R^n}|W_s|^2(y,s)\rho(y)\,dy\,ds,
\end{aligned}$$
hence \eqref{estWs} guarantees 
\begin{equation} \label{estWs2}
\int_{\sigma_1}^{\sigma_2}\int_{\R^n}|W_s|^2(y,s)\rho(y)\,dy\,ds
\leq 2C_1\Bigl(|{\mathcal E}(\sigma_2)|+|{\mathcal E}(\sigma_1)|
+\int_{\R^n}|W|^2(y,\sigma_1)\rho(y)\,dy\Bigr).
\end{equation}

Notice that \eqref{bound-u} guarantees
\begin{equation} \label{fUp}
 |f(U)\cdot U|+|F(U)|\leq C,
\end{equation}
hence \eqref{bound-w2} and the definitions of $E(s)$ and $\phi$ imply
$|E(s)|\leq Ck^{(p+1)\beta}$ for $s\in[s_k-M-1,s_k+M+1]$.
Consequently,
there exists $\tilde\sigma\in[s_k-(M+2)d,s_k-(M+1)d]$  such that
$$ 2dE(\tilde\sigma)\leq 2\int_{s_k-(M+2)d}^{s_k-(M+1)d}E(s)\,ds \leq Ck^{(p+1)\beta},$$
and similarly there exists
$\hat\sigma\in[s_k+Md,s_k+(M+1)d]$ such that
$$ 2dE(\hat\sigma)\geq-Ck^{(p+1)\beta}.$$
These estimates, integration of \eqref{GK2} and estimates \eqref{fUp}, \eqref{bound-w2} imply
$$
\int_{\tilde\sigma}^{\hat\sigma}\int_{\R^n}|W_s|^2(y,s)\rho(y)\,dy\,ds\leq Ck^{(p+1)\beta},
$$
hence
\begin{equation}\label{EM}
 \int_{s_k-(M+1)d}^{s_k+Md}\int_{\R^n}|W_s|^2(y,s)\rho(y)\,dy\,ds\leq Ck^{(p+1)\beta}.
\end{equation}

\goodbreak
{\bf Step 2: The plan of the proof.}

We divide the interval $[s_k-(M+1)d,s_k+Md]$
into $2M+1$ intervals 
$$I_j:=[s_k-(M+1-j)d,s_k-(M-j)d],\quad j=0,\dots,2M,$$
and we will use a bootstrap to improve estimate \eqref{EM}.
In the $m$-th bootstrap step ($m=1,2,\dots,M$) we will find 
$\sigma_j^m\in I_j$, $j=m-1,m,\dots,2M-m+1$, such that
\begin{equation} \label{boot1}
|{\mathcal E}(\sigma_j^m)|+\int_{\R^n}|W|^2(y,\sigma_j^m)\rho(y)\,dy\leq
Ck^{\gamma_m}, \quad j=m-1,m,\dots,2M-m+1,
\end{equation}
where 
$$ (p+1)\beta=:\gamma_0>\gamma_1>\gamma_2>\dots>\gamma_M, \quad
   \gamma_M<\mu:=2\beta-\frac{n-2}2. $$
This and \eqref{estWs2} yield
\begin{equation} \label{boot2}
 \int_{\sigma_j^m}^{\sigma_{j+1}^m}\int_{\R^n}|W_s|^2(y,s)\rho(y)\,dy\,ds\leq Ck^{\gamma_m},
\quad j=m-1,m,\dots,2M-m,
\end{equation}
hence
\begin{equation} \label{E}
 \int_{s_k-(M+1-m)d}^{s_k+(M-m)d}\int_{\R^n}|(W^a_k)_s|^2(y,s)\rho(y)\,dy\,ds\leq Ck^{\gamma_m}, 
\quad a\in\R^n,\ k\hbox{ large},
\end{equation}
where $m=1,2,\dots,M$,
and ``$k$ large'' means $k\geq k_0$ with $k_0=k_0(n,p,D,\eta,f,\fzero,U)$. 
Then, taking $\lambda_k:=k^{-1/2}$ and setting
\begin{equation} \label{Vk-1}
 V_k(z,\tau):=\lambda_k^{2\beta}W_k^a(\lambda_k z,\lambda_k^2\tau+s_k),
 \qquad z\in\R^n,\ -kd\leq\tau\leq0, 
\end{equation}
where $a=0$, 
we obtain
$|V_k|+|\nabla V_k|\leq C$, $V_k(0,0)=U(0,0)\ne0$,
$$ \frac{\partial V_k}{\partial\tau}-\Delta V_k-(\lambda_k^{2\beta}\phi^{-1})^pf(\lambda_k^{-2\beta}\phi V_k)
  =-\lambda_k^2\Bigl(\frac12 z\cdot\nabla V_k+\beta V_k\Bigr). $$
Notice that $\lambda_k^{-2\beta}\phi(s)=\phi(s-s_k)=e^{\beta\tau/k}\to1$ for fixed $\tau$.
Using \eqref{E} with $m=M$  we also have
\begin{equation} \label{estvtau}
\begin{aligned}
\int_{-k/d}^0\int_{|z|<\sqrt{k}}
 &\Big|\frac{\partial V_k}{\partial\tau}(z,\tau)\Big|^2\,dz\,d\tau
  =\lambda_k^{2\mu}
 \int_{s_k-d}^{s_k}\int_{|y|<1} 
\Big|\frac{\partial W_k^0}{\partial s}(y,s)\Big|^2\,dy\,ds \\
&\leq C k^{-\mu+\gamma_M}\to 0 \quad\hbox{as }\ k\to\infty.
\end{aligned}
\end{equation}
Now the same arguments as in \cite[pp.~18-19]{GK87} 
(cf.~also \cite[Remark 3]{Q21} and \cite{Q22})
show that
(up to a subsequence) the sequence $\{V_k\}$
converges to a nontrivial solution $V=V(z)\in{\mathcal K}$
of the problem $\Delta V+f(V)=0$ in $\R^n$,
which contradicts our assumption and
concludes the proof.

In our bootstrap argument below we fix $m\in\{1,2,\dots,M\}$, we assume that \eqref{E} is true 
with $m$ and $\gamma_m$ replaced by $m-1$ and $\gamma_{m-1}\in[\mu,\gamma_0]$, respectively 
(notice that this is true if $m=1$ due to \eqref{EM})
and we will show that \eqref{E} remains true with parameters $m$ and $\gamma_m$, 
provided $\gamma_{m}=\gamma_{m-1}-\delta$, where
$\delta=\delta(n,p)>0$ is small enough.
In other words, we will assume
\begin{equation} \label{Em}
 \int_{s_k-(M+2-m)d}^{s_k+(M+1-m)d}\int_{\R^n}|(W^a_k)_s|^2(y,s)\rho(y)\,dy\,ds\leq Ck^{\gamma}, 
\quad a\in\R^n,\ k\hbox{ large}
\end{equation}
(where $\gamma:=\gamma_{m-1}\in[\mu,\gamma_0]$)
and we will show that
\begin{equation} \label{Em1}
 \int_{s_k-(M+1-m)d}^{s_k+(M-m)d}\int_{\R^n}|(W^a_k)_s|^2(y,s)\rho(y)\,dy\,ds\leq Ck^{\gamma-\delta}.
\quad a\in\R^n,\ k\hbox{ large}
\end{equation}
Consequently, we can find $M=M(n,p)$ such that $\gamma_M<\mu\leq\gamma_{M-1}$:
In fact, $M$ is the smallest integer satsifying $M>n/(2\delta)$.

{\bf Step 3: Notation and auxiliary results.}
In the rest of the proof we will also use the following notation
and facts:
Set 
$$C(M):=8ne^{M+1}, \ \ B_r(a):=\{x\in\R^n:|x-a|\leq r\}, \ \ B_r:=B_r(0), \ \ 
   R_k:=\sqrt{8n\log k}.$$
Given $a\in\R^n$, there exists an integer $X=X(k)$ and
there exist $a^1,a^2,\dots a^X\in\R^n$ (depending on $a,n,p,k$) such that
$a^1=a$, $X\leq C(\log k)^{n/2}$ and
\begin{equation} \label{B}
 D^k(a):=B_{\sqrt{C(M)k\log(k)}}(a)\subset\bigcup_{i=1}^X B_{e^{-M/8}\sqrt{k}/2}(a^i).
\end{equation}
Notice that if $y\in B_{R_k}$ and $s\in[s_k-(M+1)d,s_k+Md]$,
then $a+ye^{-s/2}\in D^k(a)$, hence
\eqref{B} guarantees the existence of $i\in\{1,2,\dots,X\}$
such that 
\begin{equation} \label{yyi}
W^a_k(y,s)=W^{a^i}_k(y^i,s),\qquad\hbox{where}\quad
 y^i:=y+(a-a^i)e^{s/2}\in B_{1/2}.
\end{equation}

The contradiction argument in Step~2 based on 
the nonexistence of positive stationary solutions of \eqref{eq-U}
and on estimate
\eqref{estvtau}, combined with a doubling argument,
can also be used to obtain the following
useful pointwise estimates of the solution $u$
assuming a suitable $L^2$ bound on the time derivative in a cylinder.

\begin{lemma} \label{lem-decay}
Let $M,d,s_k,W^a_k$ be as above, $\zeta\in\R$, $\xi,C^*>0$,
$d_k\in(0,d]$ and $r_k\in(0,1]$, $k=1,2,\dots$.
Set
$$ {\mathcal T}_k:=\Bigl\{(a,\sigma,b)\in\R^n\times[s_k-Md,s_k+Md]\times\R^n:
  \int_{\sigma-d_k}^{\sigma}\int_{B_{r_k}(b)} |(W^a_k)_s|^2 dy\,ds \leq C^*k^\zeta
\Bigr\}. $$ 
Assume
\begin{equation} \label{xizeta}
 \xi\frac\mu\beta>\zeta\quad{and}\quad
 \frac1{\log(k)}\min(d_kk^{\xi/\beta_k},r_kk^{\xi/2\beta_k})\to\infty\ 
 \hbox{ as }\ k\to\infty,
\end{equation}
whenever $\beta_k>0$, $\beta_k\to\beta$ as $k\to\infty$.
Then there exists $\tilde k_0$ such that
$$ (|W^a_k|+|\nabla W^a_k|^{2/(p+1)})(y,\sigma)\leq k^\xi \quad\hbox{whenever}\quad
 y\in B_{r_k/2}(b), \ \  k\geq \tilde k_0 \ 
\hbox{ and } \ (a,\sigma,b)\in{\mathcal T}_k.$$ 
\end{lemma}

\begin{proof}
The proof is a modification of the proof of
\cite[Lemma~10]{Q22}. 
Assume to the contrary that there exist $k_1,k_2\dots$ with the following properties:
$k_j\to\infty$ as $j\to\infty$, and for each $k\in\{k_1,k_2,\dots\}$ there exist
$(a_k,\sigma_k,b_k)\in{\mathcal T}_k$ and $y_k\in B_{r_k/2}(b_k)$ such that
$\tilde w_k(y_k,\sigma_k)>k^\xi$, where 
$\tilde w_k:=|W^{a_k}_k|+|\nabla W^{a_k}_k|^{2/(p+1)}$.

Given $k\in\{k_1,k_2\,\dots\}$, we can choose a positive integer $K$ such that
\begin{equation} \label{Klemma}
  2^Kk^{\xi}>2C_0k^\beta, \qquad K<C\log k.
\end{equation}
Set
$$ Z_q:=B_{r_k(1/2+q/(2K))}(b_k)\times[\sigma_k-d_k(1/2+q/(2K)),\sigma_k],
 \quad q=0,1,\dots,K.$$
Then
$$ B_{r_k/2}(b_k)\times[\sigma_k-d_k/2,\sigma_k]=Z_0\subset Z_1\subset\dots\subset Z_K=B_{r_k}(b_k)\times[\sigma_k-d_k,\sigma_k].$$
Since $\sup_{Z_0}\tilde w_k\geq \tilde w_k(y_k,\sigma_k)>k^{\xi}$,
estimates \eqref{Klemma} and \eqref{bound-w2} imply the existence of $q^*\in\{0,1,\dots K-1\}$ such that
$$ 2\sup_{Z_{q^*}}\tilde w_k \geq \sup_{Z_{q^*+1}}\tilde w_k $$
(otherwise $2C_0k^\beta\geq\sup_{Z_K}\tilde w_k>2^K\sup_{Z_0}\tilde w_k>2^K k^\xi$, a contradiction).
Fix $(\hat y_k,\hat s_k)\in Z_{q^*}$ such that
$$ \tilde W_k:=\tilde w_k(\hat y_k,\hat s_k)=\sup_{Z_{q^*}}\tilde w_k.$$
Then $\tilde W_k\geq  k^{\xi}$,
$$ \hat Q_k:=B_{r_k/(2K)}(\hat y_k)\times\Bigl[\hat s_k-\frac{d_k}{2K},\hat s_k\Bigr]\subset Z_{q^*+1},$$
and $\tilde w_k\leq 2\tilde W_k$ on $\hat Q_k$.

Set $z_k:=\tilde W_k\phi(\hat s_k)$.
Since $\tilde W_k\leq Ck^{\beta}$ due to \eqref{bound-w2}
and $\phi(\hat s_k)\leq Ck^{-\beta}$, we have either
$z_k\to 0$ or
there exists a subsequence (which will be labeled the same) such that
$z_k\to c_\lambda>0$.

If $z_k\to c_\lambda>0$, then
we set $\lambda_k:=(\tilde W_k/c_\lambda)^{-1/(2\beta_k)}$, where $\beta_k:=\beta$
(hence $\lambda_k^{-2\beta_k}\phi(\hat s_k)=z_k/c_\lambda\to1$ and $\lambda_k\to0$).

If $z_k\to0$, then we set $c_\lambda:=1$,
$$ \lambda_k:=\tilde W_k^{-1/(2\beta)}\Bigl(\frac{z_k^p}{f^+(z_k)}\Bigr)^{1/2}$$
and $\beta_k$ is defined by $\lambda_k=\tilde W_k^{-1/(2\beta_k)}$
(hence $\lambda_k^{-2\beta_k}\phi(\hat s_k)=z_k\to0$).
Notice that \eqref{fp0} with $\theta$ small enough and the inequalities
$\tilde W_k\geq k^\xi$, $\phi(\hat s_k)\geq ck^{-\beta}$ imply $\lambda_k\to0$.
We also have $\beta_k\to\beta$. In fact, assume to the contrary
that $\alpha_k:=\frac1\beta-\frac1{\beta_k}\not\to0$,
We may assume $|\alpha_k|\geq \alpha$ for some $\alpha>0$.
Choose $\theta>0$ such that $\theta(\beta-\xi)<\xi\alpha$. Then \eqref{fp0} implies
$$ \tilde W_k^{\alpha_k}=\frac{z_k^p}{f^+(z_k)}\in[z_k^\theta,z_k^{-\theta}]
\quad\hbox{ for $k$ large,} $$
which contradicts the inequalities
$\tilde W_k\geq k^\xi$ and $ck^{-\beta}\leq\phi(\hat s_k)\leq Ck^{-\beta}$.

Set 
\begin{equation} \label{Vk-2}
 V_k(z,\tau):=\lambda_k^{2\beta_k}W^{a_k}_k(\lambda_k z+\hat y_k,\lambda_k^2\tau+\hat s_k).
\end{equation}
Then $V_k(\cdot,\tau)\in{\mathcal K}$,
$(|V_k|+|\nabla V_k|^{2/(p+1)})(0,0)=c_\lambda$, 
 $|V_k|+|\nabla V_k|^{2/(p+1)}\leq2c_\lambda$ 
on $Q_k:=B_{r_k/(2K\lambda_k)}\times[-d_k/(2K\lambda_k^2),0]$, and
\begin{equation} \label{eqvk} 
  \frac{\partial V_k}{\partial\tau}-\Delta V_k
  -\lambda_k^{2\beta_k+2}\phi^{-p}f(\lambda_k^{-2\beta_k}\phi V_k)
  =-\lambda_k^2\Bigl(\frac12 z\cdot\nabla V_k+\beta V_k\Bigr) \quad\hbox{on }\ Q_k,
\end{equation}
where $\phi=\phi(s)=e^{s/(p-1)}=\phi(\hat s_k)e^{\tau\lambda_k^2/p}\to \phi(\hat s_k)$
for fixed $\tau$.
Notice also that
$$\lambda_k^{2\beta_k+2}\phi^{-p}(\hat s_k)=[f^+(\lambda_k^{-2\beta_k}\phi(\hat s_k)]^{-1}.$$
In addition, as $k\to\infty$,
$$
  \frac{r_k}{2K\lambda_k} \geq\frac{r_k k^{\xi/(2\beta_k)}}{C\log(k)}\to\infty, \quad
  \frac{d_k}{2K\lambda_k^2} \geq\frac{d_k k^{\xi/\beta_k}}{C\log(k)}\to\infty.
$$
Since $(a_k,\sigma_k,b_k)\in{\mathcal T}_k$ and $\hat Q_k\subset Z_K$, we obtain
$$
\int_{Q_k}\Big|\frac{\partial V_k}{\partial\tau}(z,\tau)\Big|^2\,dz\,d\tau
  =\lambda_k^{2\mu}\int_{\hat Q_k}
\Big|\frac{\partial W^{a_k}_k}{\partial s}(y,s)\Big|^2\,dy\,ds \leq C^*k^\delta,
\quad\hbox{where }\ \delta:=-\xi\frac\mu\beta+\zeta <0.
$$
Hence, as above, a suitable subsequence of $\{V_k\}$
converges to a nontrivial solution $V\in{\mathcal K}$
of either the problem $\Delta V+\fzero(V)=0$ in $\R^n$ 
(if $\lambda_k^{-2\beta_k}\phi(\hat s_k)\to0$; cf.~\eqref{F1}),
or the problem $\Delta V+f(V)=0$ in $\R^n$ 
(if $\lambda_k^{-2\beta_k}\phi(\hat s_k)\to1$).
This contradicts our assumptions.
\end{proof}

Recall that $\gamma\in[\mu,(p+1)\beta]$.

\begin{Lemma} \label{lemmaG}
Let ${\mathcal T}_k={\mathcal T}_k(d_k,r_k,\zeta,C^*)$ be as in Lemma~\ref{lem-decay},
$\omega\in\R$, $\eps,C^*>0$,
\begin{equation} \label{ass-rep}
 0\leq\alpha<\frac\xi\beta,\qquad \xi\frac\mu\beta>\gamma-\alpha+\eps-\omega, 
\end{equation}
and assume that 
\begin{equation} \label{assG}
 \hbox{$(a,\sigma,0)\in {\mathcal T}_k(\frac d2 k^{-\alpha},1,\gamma-\alpha+\eps,C^*)$\ \ for $k$ large}.
\end{equation}
Set 
$$G:=\{y\in B_{1/2}: (|W^a_k|+|\nabla W^a_k|^{2/(p+1)})(y,\sigma)\leq k^\xi\}.$$
Then 
\begin{equation} \label{Gc}
|B_{1/2}\setminus G|\leq Ck^{\omega-n\alpha/2}\ \hbox{ for $k$ large}.
\end{equation}
\end{Lemma}

\begin{proof}
There exist 
$b^1,\dots,b^{Y}\in B_{1/2}$ with $Y\leq Ck^{n\alpha/2}$
such that 
$$ B_{1/2}\subset \bigcup_{j=1}^{Y}B^j,\quad\hbox{where}\quad
  B^j:=B_{\frac12k^{-\alpha/2}}(b^j), $$
and 
\begin{equation} \label{multiynew}
 \#\{j: y\in B_{k^{-\alpha/2}}(b^j)\}\leq C_n\quad\hbox{for any }\ y\in\R^n.
\end{equation}
Set
$$ \begin{aligned}
  H &:=\bigl\{j\in\{1,2,\dots,Y\}: 
  (a,\sigma,b^j)\in{\mathcal T}_k(\hbox{$\frac d2$}k^{-\alpha},k^{-\alpha/2},\gamma-\alpha+\eps-\omega,C^*C_n)\bigr\}, \\
  H^c &:=\{1,2,\dots,Y\}\setminus H.
\end{aligned}
$$
If $j\in H$, then \eqref{ass-rep} implies 
\eqref{xizeta}, hence Lemma~\ref{lem-decay} 
guarantees 
$$(|W^a_k|+|\nabla W^a_k|^{2/(p+1)})(y,\sigma)\leq k^\xi\ \hbox{ for }\ y\in B^j.$$
Consequently, 
$$B_{1/2}\cap\bigcup_{j\in H}B^j \subset G, \ \hbox{ hence } \
  B_{1/2}\setminus G \subset \bigcup_{j\in H^c}B^j. $$
Now \eqref{assG}, the definition of $H$ and \eqref{multiynew} imply
$\#H^c< Ck^\omega$, 
hence \eqref{Gc} is true.
\end{proof}

\begin{Lemma} \label{lem-ineq}
Fix a positive integer $L=L(n,p)$ such that
\begin{equation} \label{estL}
 \frac\mu{p+1}\Bigl(\frac{2p}{p+1}\Bigr)^L>\beta.
\end{equation}
There exists $\delta=\delta(n,p)>0$ with the following properties:
If $\gamma\in[\mu,(p+1)\beta]$ and $\eps>0$ is small enough, then there exist
$\xi_\ell,\alpha_\ell,\omega_\ell$, $\ell=1,2,\dots L$, 
such that 
\begin{equation} \label{estxi}
\xi_1\leq\xi_2\leq\dots\leq\xi_L\leq\beta=:\xi_{L+1},
\end{equation}
$(p+1)\xi_1\leq\gamma-2\delta$,
and the following inequalities are true for $\ell=1,2,\dots,L$:
\begin{equation} \label{bootcond}
0\leq\alpha_\ell <\frac{\xi_\ell}\beta, \quad \xi_\ell\frac\mu\beta>\gamma-\alpha_\ell+\eps-\omega_\ell, \quad
\omega_\ell-\frac{n\alpha_\ell}2\leq\gamma-2\delta-(p+1)\xi_{\ell+1}.
\end{equation}
\end{Lemma}

\begin{proof}
Considering $\alpha_\ell$ close to (and less than) $\xi_\ell/\beta$ and $\eps>0$ small,
we see that we only have to satisfy the condition $(p+1)\xi_1\leq\gamma-2\delta$ 
and the following inequalities for $\ell=1,2,\dots,L$:
\begin{equation} \label{bootcond2}
 \xi_\ell\frac\mu\beta>\gamma-\frac{\xi_\ell}\beta-\omega_\ell, \quad
\omega_\ell-\frac{n}2\frac{\xi_\ell}\beta<\gamma-2\delta-(p+1)\xi_{\ell+1}.
\end{equation}
One can find $\omega_\ell$ such that the inequalities \eqref{bootcond2} are true,
provided the lower bound for $\omega_\ell$ in these inequalities
is less than the upper bound, i.e. if
$$ \gamma-\frac{\xi_\ell}\beta(1+\mu)<\frac{n}2\frac{\xi_\ell}\beta+\gamma-2\delta-(p+1)\xi_{\ell+1}, $$
which is equivalent to
$$ (p+1)\xi_{\ell+1}<2p\xi_\ell-2\delta. $$
Consider $\delta\leq\frac{p-1}{5p+1}\mu\leq\frac13$ and set
$$\xi_1=\xi_1(\gamma):=\frac{\gamma-2\delta}{p+1}\geq\frac{\mu-2\delta}{p+1},
\quad \xi_{\ell+1}:=\min\Bigl(\beta,\frac{2p\xi_\ell-3\delta}{p+1}\Bigr),\quad\ell=1,\dots,L. $$
Then $\xi_\ell=\xi_\ell(\gamma)\geq\xi_\ell(\mu)$
and $0<\xi_1\leq\xi_2\leq\dots\leq\xi_{L+1}$.
Now \eqref{estL} guarantees that $\xi_{L+1}=\beta$ provided
$\delta=\delta(n,p)$ is small enough.
\end{proof}

The conclusion of Step~3, that will be used in Step~5, is the following lemma,
which is a consequence of Lemmas~\ref{lem-decay}--\ref{lem-ineq}.
It provides an improved energy bound assuming a suitable $L^2$ bound
of the time derivative in a cylinder.

\begin{Lemma} \label{lem-wq}
Let $L,\delta,\eps$ and 
$\xi_\ell,\alpha_\ell,\omega_\ell$, $\ell=1,2,\dots L$, be as in
Lemma~\ref{lem-ineq} and let
${\mathcal T}_k={\mathcal T}_k(d_k,r_k,\zeta,C^*)$ be as in Lemma~\ref{lem-decay}.
Assume $(a,\sigma,0)\in {\mathcal T}_k(\frac d2 k^{-\alpha_\ell},1,\gamma-\alpha_\ell+\eps,C)$
for $\ell=1,2,\dots L$ and $k$ large. 
Then 
\begin{equation} \label{est-wq}
\int_{B_{1/2}}(|\nabla W^a_k|^2+|W^a_k|^2+|W^a_k|^{p+1})(y,\sigma)\,dy\leq Ck^{\gamma-2\delta}\ \hbox{ for $k$ large}.
\end{equation}
\end{Lemma}

\begin{proof}
Given $\ell\in\{1,2,\dots,L\}$, set $\xi=\xi_\ell$, $\alpha=\alpha_\ell$, $\omega=\omega_\ell$,
and let $G$ be the set in Lemma~\ref{lemmaG}.
Set $G_\ell:=G$ and $G_{L+1}:=B_{1/2}$. 
Lemma~\ref{lemmaG} and \eqref{bootcond} guarantee
$$|G_{\ell+1}\setminus G_\ell|\leq|B_{1/2}\setminus G_\ell|\leq Ck^{\omega_\ell-n\alpha_\ell/2}
                 \leq Ck^{\gamma-2\delta-(p+1)\xi_{\ell+1}},$$
hence
$$\int_{G_{\ell+1}\setminus G_\ell}(|\nabla W^a_k|^2+|W^a_k|^2+|W^a_k|^{p+1})(y,\sigma)\,dy
 \leq Ck^{(p+1)\xi_{\ell+1}}|G_{\ell+1}\setminus G_\ell| \leq Ck^{\gamma-2\delta}.$$
In addition, the definition of $G_1$ implies
\begin{equation} \label{estG1}
 \int_{G_1}(|\nabla W^a_k|^2+|W^a_k|^2+|W^a_k|^{p+1})(y,\sigma)\,dy 
 \leq Ck^{(p+1)\xi_1} \leq Ck^{\gamma-2\delta}.
\end{equation}
Since $B_{1/2}=G_1\cup\bigcup_{\ell=1}^L (G_{\ell+1}\setminus G_\ell)$,
the conclusion follows.
\end{proof}

\smallskip\noindent{\bf Step 4: The choice of suitable times.}
The proof of \eqref{Em1} will be based on estimates of $W^{a^i}_k(\cdot,s^*)$,
$i=1,2,\dots,X$, where
$s^*=s^*(k,a,j,m)\in I_j=[s_k-(M+1-j)d,s_k-(M-j)d]$, $j=m-1,m,\dots,2M-m+1$
are suitable times. 
These estimates will follow from Lemma~\ref{lem-wq},
provided we can verify its assumptions simultaneously for all $a^i$
(namely, the $L^2$ bound on the time derivative in a cylinder).
To this end, we shall use 
the following lemma. 

\begin{Lemma} \label{lem-step4}  
Let $\eps,\gamma,C_2>0$, $\alpha_1,\alpha_2,\dots,\alpha_L\geq0$,
and, given $k=1,2,\dots$, let $X_k$ be a positive integer satisfying 
$X_k\leq k^{\eps/2}$ and $\sigma_k\in\R$. Set $J_k:=[\sigma_k,\sigma_k+d]$, 
$\tilde J_k:=[\sigma_k+d/2,\sigma_k+d]$, and
assume that  $f^1_k,\dots,f^{X_k}_k\in C(J_k,\R^+)$ 
satisfy 
\begin{equation} \label{figi}
\int_{J_k}f^i_k(s)\,ds\leq C_2k^\gamma, \quad 
i=1,2,\dots X_k,\ k=1,2,\dots.
\end{equation}
Then there exists $k_1=k_1(\eps,L)$ with the following property:
If $k\geq k_1$, then there exists $s^*=s^*(k)\in\tilde J_k$ such that 
$$ \int_{s^*-\frac d2k^{-\alpha_\ell}}^{s^*}f^i_k(s)\,ds\leq C_2k^{\gamma-\alpha_\ell+\eps} $$ 
for all $i=1,2,\dots,X_k$ and $\ell=1,2,\dots,L$.
\end{Lemma}

\begin{proof}
Set
$$h^{i,\ell}_k(s):=\int_{s-\frac d2k^{-\alpha_\ell}}^{s}f^i_k(\tau)\,d\tau, \ \ s\in\tilde J_k,\ \
i=1,2,\dots X_k,\ \ \ell=1,2,\dots,L,\ \ k=1,2,\dots.$$
Then 
\begin{equation} \label{hiell}
\begin{aligned}
\int_{\tilde J_k} &h^{i,\ell}_k(s)\,ds
  = \int_{\tilde J_k}\int_{s-\frac d2k^{-\alpha_\ell}}^{s}f^i_k(\tau)\,d\tau\,ds
  = \int_{\tilde J_k}\int_0^{\frac d2k^{-\alpha_\ell}}f^i_k(s-\tau)\,d\tau\,ds \\
 &= \int_0^{\frac d2k^{-\alpha_\ell}}\int_{\tilde J_k}f^i_k(s-\tau)\,ds\,d\tau
  \leq \int_0^{\frac d2k^{-\alpha_\ell}}\int_{J_k} f^i_k(s)\,ds\,d\tau\leq C_2k^{\gamma-\alpha_\ell}.
\end{aligned}
\end{equation}
Set
$$ \begin{aligned}
   C^{i,\ell}_k &:=\{s\in\tilde J_k: h^{i,\ell}_k(s)>C_2k^{\gamma-\alpha_\ell+\eps}\}.
\end{aligned}$$
Then \eqref{hiell} implies $|C^{i,\ell}_k|\leq k^{-\eps}$ for each $i,\ell,k$.
Since the number of sets $C^{i,\ell}_k$ with fixed index $k$ is $LX_k\leq Lk^{\eps/2}$,
their union
$U_k:=\bigcup_{i,\ell}C^{i,\ell}_k$
has measure less than 1/2 for $k\geq k_1$, hence for $k\geq k_1$ 
there exists $s^*=s^*(k)\in\tilde J_k\setminus U_k$.
Obviously, $s^*$ has the required properties.
\end{proof}

Consider $m\in\{1,\dots,M-1\}$, $\gamma\in[\mu,(p+1)\beta]$ and $a\in\R^n$ fixed,
$J_k=J_k(j):=I_j=[s_k-(M+1-j)d,s_k-(M-j)d]$, $j=m-1,m,\dots,2M-m+1$,
and let $a^i$, $i=1,2,\dots,X$ be as in \eqref{B} (recall that $a^i$ and $X$ depend on $k$;
$X\leq C(\log k)^{n/2}$).
Let $L,\eps$ and $\alpha_\ell$, $\ell=1,2,\dots,L$ be as in Lemma~\ref{lem-ineq}.
Set 
$$f^i_k(s):=\int_{\R^n}|(W^{a^i}_k)_s|^2(y,s)\rho(y)\,dy,\quad
\quad i=1,2,\dots X.$$
Then \eqref{Em} guarantees
that the assumptions of  Lemma~\ref{lem-step4} are satisfied
with $C_2$ independent of $a$.
Consequently, if $k\geq k_1$, $a\in\R^n$ and $j\in\{m-1,m,\dots,2M-m+1\}$,  then
there exists $s^*=s^*(k,a,j,m)\in\tilde J_k:=[s_k-(M-j+1/2)d,s_k-(M-j)d]$ 
such that the following estimates are true  
for $W=W_k^{a^i}$, $i=1,2,\dots X$, $\ell=1,2,\dots L$: 
\begin{equation} \label{star1}
 \int_{s^*-\frac d2k^{-\alpha_\ell}}^{s^*}\int_{\R^n}|W_s|^2(y,s)\rho(y)\,dy\,ds \leq C_2k^{\gamma-\alpha_\ell+\eps}.
\end{equation}

\smallskip\noindent{\bf Step 5: Energy estimates.}
Let $L,\delta,\eps$ and 
$\xi_\ell,\alpha_\ell,\omega_\ell$, $\ell=1,2,\dots L$, be as in
Lemma~\ref{lem-ineq} and let
${\mathcal T}_k={\mathcal T}_k(d_k,r_k,\zeta,C^*)$ be as in Lemma~\ref{lem-decay}.
Let $a\in\R^n$ be fixed, $j\in\{m-1,m,\dots,2M-m+1\}$ and $s^*=s^*(k,a,j,m)$ be from Step~4.
Notice that \eqref{star1} guarantees 
$(a^i,s^*,0)\in{\mathcal T}_k(\frac d2k^{-\alpha_\ell},1,\gamma-\alpha_\ell+\eps,C_2/\rho(1))$
for $i=1,2,\dots,X$ and $\ell=1,2,\dots,L$.
Consequently, Lemma~\ref{lem-wq} implies
$$  \int_{B_{1/2}}(|\nabla W^{a^i}_k|^2+|W^{a^i}_k|^2+|W^{a^i}_k|^{p+1})(y,s^*)\,dy
\leq Ck^{\gamma-2\delta}\   \hbox{ for $i=1,2,\dots,X$ and $k$ large}, $$ 
and using \eqref{yyi} we obtain
$$ \begin{aligned}
\int_{B_{R_k}}&(|\nabla W^a_k|^2+|W^a_k|^2+|W^a_k|^{p+1})(y,s^*)\,dy \\
 &\leq \sum_{i=1}^X\int_{B_{1/2}}(|\nabla W^{a^i}_k|^2+|W^{a^i}_k|^2+|W^{a^i}_k|^{p+1})(y,s^*)\,dy \\
 &\leq Ck^{\gamma-2\delta}(\log k)^{n/2}
 \leq Ck^{\gamma-3\delta/2}\quad  \hbox{ for $k$ large}.
\end{aligned} $$ 
In addition, since
$$  \rho(y)=e^{-|y|^2/8-|y|^2/8}\leq k^{-n}e^{-|y|^2/8}\quad\hbox{for }\ |y|>R_k,$$
we have
$$\begin{aligned}
\int_{\R^n\setminus B_{R_k}}&(|\nabla W^a_k|^2+|W^a_k|^2+|W^a_k|^{p+1})(y,s^*)\rho(y)\,dy \\
 &\leq C\int_{\R^n\setminus B_{R_k}}k^{(p+1)\beta-n}e^{-|y|^2/8}\,dy 
 \leq Ck^{(p+1)\beta-n} = Ck^{\mu-n/2} \leq C k^{\gamma-3\delta/2} 
\end{aligned}
$$
due to $\gamma\geq\mu$,
hence
\begin{equation} \label{est32}
 \int_{\R^n}(|\nabla W^a_k|^2+|W^a_k|^2+|W^a_k|^{p+1})(y,s^*)\rho(y)\,dy
 \leq C k^{\gamma-3\delta/2}
\end{equation}
for $k$ large.
Since \eqref{fp0} with $\theta:=\delta/(2\beta)$ implies
$|F(U)|\leq C|U|^{p+1-\theta}$, estimate \eqref{est32} guarantees
\begin{equation} \label{est-fin}
\begin{aligned}
\Big|\phi^{-(p+1)}(s^*)\int_{\R^n}F(\phi W^a_k)(y,s^*)\rho(y)\,dy\Big|
&\leq C\phi^{-\beta\theta}(s^*)\int_{\R^n}|W^a_k|^{p+1-\theta}(y,s^*)\rho(y)\,dy \\
&\leq Ck^{\beta\theta+(\gamma-3\delta/2)(p+1-\theta)/(p+1)}\leq Ck^{\gamma-\delta}.
\end{aligned}
\end{equation}
Estimates \eqref{est32} and \eqref{est-fin} imply \eqref{boot1} 
with $\sigma_j^m$ replaced by $s^*(k,a,j,m)$ and
$\gamma_m$ replaced by $\gamma-\delta$,
hence also estimates \eqref{boot2}, \eqref{E} with $\gamma_m$ replaced by $\gamma-\delta$.
This proves \eqref{Em1}
and concludes the proof.
\qed

\section{Proof of Theorem~\ref{thmU-half}} 
\label{sec-proof-half}

Theorem~\ref{thmU-half} can be proved by the same modification of the proof
of Theorem~\ref{thmU} as in the case of 
\cite[the proof of Theorem~3]{Q22}.
Alternatively, we can proceed as follows.

Set $\tilde K:=K\cup(-K)$,
$$ \tilde{\mathcal K}:=\{U\in C^2(\R^n,\R^N):u_i\geq0\hbox{ in }\R^n_+,\ 
  U(x_1,\dots,x_{n-1},x_n)=-U(x_1,\dots,x_{n-1},-x_n)\},$$
and extend $F$ and $U$ by $F(V):=F(-V)$ for $V\in-K$
and $U(x_1,\dots,x_{n-1},x_n):=-U(x_1,\dots,x_{n-1},-x_n)$ for $x_n<0$
and $U\in {\mathcal K}_+$, respectively.
(Notice that $F$ remains $C^{1+\alpha}$ since our assumptions imply $\nabla F(0)=f(0)=0$.)
Then we can apply the proof of Theorem~\ref{thmU} with $K,{\mathcal K}$ 
replaced by $\tilde K,\tilde{\mathcal K}$  
and with the following modifications:

In \eqref{Vk-1}, we choose $a\in\R^n_+$ such that $U(a,0)\neq0$
(or set $a:=0$ and use $|U(0,0)|+|\nabla U(0,0)|\neq0$).
In \eqref{Vk-2}, we modify the definition of $V_k$ as follows.
Let $\lambda_k,\beta_k,a_k,\hat y_k,\hat s_k$ be as in \eqref{Vk-2}.
Set
$$ \eta_k:=\frac1{\lambda_k}(\hat y_{k,n}+a_{k,n}e^{\hat s_k/2}),$$
where $\hat y_k=(\hat y_{k,1},\dots,\hat y_{k,n})$
and $a_k=(a_{k,1},\dots,a_{k,n})$.
W.l.o.g. we can assume that $\eta_k\geq0$ for all $k$
and either $\eta_k\to\infty$ or $\eta_k\to\eta_\infty\in[0,\infty)$.
If $\eta_k\to\infty$, then the definition of $V_k$ in \eqref{Vk-2}
need not be modified and we obtain $V_k\to V$, 
where $V(\cdot,\tau)=V(\cdot,0)\in {\mathcal K}$.
If $\eta_k\to\eta_\infty<\infty$, then setting
$$\tilde V_k(z,\tau):=V_k(z-(0,\dots,0,\eta_k),\tau)$$
we have
$$\tilde V_k(z,0)=\lambda_k^{2\beta_k}e^{-\beta\hat s_k}
 U(x,k-e^{-\hat s_k}),$$
where 
$$ x=x(z):=(\lambda_k z+(\hat y_{k,1},\dots,\hat y_{k,n-1},0))e^{-\hat s_k/2}
   +(a_{k,1},\dots,a_{k,n-1},0). $$ 
Since $x_n(z_1,\dots,z_{n-1},-z_n)=-x_n(z_1,\dots,z_{n-1},z_n)$, we have
 $\tilde V_k(\cdot,0)\in\tilde{\mathcal K}$.
In addition, $(|\tilde V_k|+|\nabla\tilde V_k|^{2/(p+1)})((0,\dots,0.\eta_k),0)=c_\lambda$
and the corresponding estimates on $|\tilde V_k|,|\nabla\tilde V_k|$ and $\partial\tilde V_k/\partial\tau$
imply that
$\tilde V_k\to\tilde V$, where $\tilde V(\cdot,\tau)=\tilde V(\cdot,0)\in\tilde{\mathcal K}$
is a nontrivial stationary solution of problem \eqref{eq-U-half}
with either $\varphi=f$ or $\varphi=\fzero$.

\section{Proof of Theorem~\ref{thmUK}} 
\label{sec-proof2}

Assume to the contrary that there exists
a nontrivial bounded solution $U$ of \eqref{eq-U} satisfying $U(\cdot,t)\in{\mathcal K}$ for all $t$.
Then there exists $D>0$ such that \eqref{bound-u} is true.
We may also assume $U(0,0)\ne0$.

Let $\beta,\mu,\phi,\rho,y,s,s_k,W=W^a_k$ and $E=E^a_k$ have the same meaning as in the proof of Theorem~\ref{thmU}
and $M:=1$.
Then \eqref{eq-w}, \eqref{bound-w2}, \eqref{GK1} and \eqref{GK2} are true.
By $C,C_0,C_1,\dots,$\allowbreak$c,c_0,c_1\dots$ we denote positive constants which depend only on
$n,p,D,f,c_D,C_D$. 
Assumption \eqref{F3} implies
$$\begin{aligned} \frac{d}{ds}&\int_{\R^n}\xi\cdot W(y,s)\rho(y)\,dy+\beta\int_{\R^n}\xi\cdot W(y,s)\rho(y)\,dy \\
&\geq c\int_{\R^n}|W(y,s)|^p\rho(y)\,dy 
\geq c\Bigl(\int_{\R^n}\xi\cdot W(y,s)\rho(y)\,dy\Bigr)^p,\end{aligned}$$
which yields (cf.~\cite{Q16},\cite{Q21})
\begin{equation} \label{estK}
\int_{\R^n}|W(y,s)|\rho(y)\,dy\leq C,\quad
\int_{\sigma_1}^{\sigma_2}\int_{\R^n}|W(y,s)|^p\rho(y)\,dy\,ds\leq C(1+\sigma_2-\sigma_1).
\end{equation}
Identity \eqref{GK2} and assumption \eqref{F3} imply
\begin{equation} \label{estsigma12}
\begin{aligned}
 \int_{\sigma_1}^{\sigma_2}\int_{\R^n}|W_s|^2(y,s)\rho(y)\,dy\,ds
 &\leq E(\sigma_1)-E(\sigma_2) \\
 &\ +C\int_{\sigma_1}^{\sigma_2}\phi^{q-p}(s)\int_{\R^n}|W(y,s)|^{q+1}\rho(y)\,dy\,ds.
\end{aligned}
\end{equation}
Notice that if $\sigma_1,\sigma_2\in[s_k-2,s_k+1]$, 
then $\phi^{-1}(s)\leq Ck^\beta$ for $s\in[\sigma_1,\sigma_2]$,
hence \eqref{estK} and \eqref{bound-w2} imply
\begin{equation} \label{betatilde}
 \int_{\sigma_1}^{\sigma_2}\phi^{q-p}(s)\int_{\R^n}|W(y,s)|^{q+1}\rho(y)\,dy\,ds
\leq Ck^{\tilde\beta},
\end{equation}
where $\tilde\beta:=\beta$ if $q\geq p-1$ and $\tilde\beta:=\beta(p-q)$ otherwise.

There exist $\tilde\sigma\in[s_k-2,s_k-1]$ and $\hat\sigma\in[s_k,s_k+1]$ such that
$E(\tilde\sigma)\leq\int_{s_k-2}^{s_k-1}E(s)\,ds$, 
$E(\hat\sigma)\geq\int_{s_k}^{s_k+1}E(s)\,ds$.
Consequently, integrating \eqref{GK1} 
we obtain from \eqref{estK}, \eqref{bound-w2} and \eqref{betatilde}
$$ \begin{aligned}
 2E(\tilde\sigma) &\leq\frac12\int_{\R^n}|W|^2(y,s_k-2)\rho(y)\,dy \\
  &\ +C\int_{s_k-2}^{s_k-1}\phi^{q-p}(s)\int_{\R^n}|W(y,s)|^{q+1}\rho(y)\,dy\,ds 
   \leq Ck^{\tilde\beta}, \\
  2E(\hat\sigma) &\geq-\frac12\int_{\R^n}|W|^2(y,s_k+1)\rho(y)\,dy \\
 &\ -C\int_{s_k}^{s_k+1}\phi^{q-p}(s)\int_{\R^n}|W(y,s)|^{q+1}\rho(y)\,dy\,ds
   \geq-Ck^{\tilde\beta}. 
\end{aligned}$$
Using these estimates 
together with \eqref{estK}, \eqref{estsigma12}, \eqref{betatilde} 
yields
\begin{equation} \label{estThm2}
 \int_{s_k-1}^{s_k}\int_{\R^n}|W_s|^2(y,s)\rho(y)\,dy\,ds 
\leq \int_{\tilde\sigma}^{\hat\sigma}\int_{\R^n}|W_s|^2(y,s)\rho(y)\,dy\,ds
\leq Ck^{\tilde\beta}.
\end{equation}
Taking $\lambda_k:=k^{-1/2}$ and setting
$$ V_k(z,\tau):=\lambda_k^{2\beta}W_k^0(\lambda_k z,\lambda_k^2\tau+s_k),
 \qquad z\in\R^n,\ -k\leq\tau\leq0,  $$ 
we obtain
$|V_k|+|\nabla V_k|\leq C$, $V_k(0,0)=U(0,0)\ne0$,
$$ \frac{\partial V_k}{\partial\tau}-\Delta V_k-(\lambda_k^{2\beta}\phi^{-1})^pf(\lambda_k^{-2\beta}\phi V_k)
  =-\lambda_k^2\Bigl(\frac12 z\cdot\nabla V_k+\beta V_k\Bigr). $$
Notice that $\lambda_k^{-2\beta}\phi(s)=\phi(s-s_k)=e^{\beta\tau/k}\to1$ for fixed $\tau$.
Using \eqref{estThm2} and $p<p_{sg}$  we also have
$$
\begin{aligned}
\int_{-k}^0\int_{|z|<\sqrt{k}}
 &\Big|\frac{\partial V_k}{\partial\tau}(z,\tau)\Big|^2\,dz\,d\tau
  =\lambda_k^{2\mu}
 \int_{s_k-1}^{s_k}\int_{|y|<1} 
\Big|\frac{\partial W_k^0}{\partial s}(y,s)\Big|^2\,dy\,ds \\
&\leq C k^{-\mu+\tilde\beta}\to 0 \quad\hbox{as }\ k\to\infty.
\end{aligned}
$$
Now the same arguments as above
show that
(up to a subsequence) the sequence $\{V_k\}$
converges to a nontrivial solution $V\in{\mathcal K}$
of the problem $\Delta V+f(V)=0$ in $\R^n$,
which contradicts our assumption and
concludes the proof.
 
\section{Proof of Theorem~\ref{thmUK2}} 
\label{sec-proof3}

Assume to the contrary that there exists
a nontrivial bounded solution $U$ of \eqref{eq-U} satisfying $U(\cdot,t)\in{\mathcal K}$ for all $t$.
Then there exists $D>0$ such that \eqref{bound-u} is true.
We may also assume $U(0,0)\ne0$.

Let $\beta,\mu,\phi,\rho,y,s,s_k,W=W^a_k,E=E^a_k,{\mathcal E}={\mathcal E}^a_k$
and $C_E$ 
have the same meaning as in the proof of Theorem~\ref{thmU}
and $M:=1$.
Then \eqref{eq-w}, \eqref{bound-w2}, \eqref{GK1}, \eqref{GK2} and \eqref{dEs} are true.
By $C,C_0,C_1,\dots,c,c_0,c_1\dots$ we denote positive constants which depend only on
$n,p,D,f,c_D,C_D$. 

Assume $\sigma_1,\sigma_2\in[s_k-2,s_k+1]$ and denote
$C_\sigma:=e^{2C_E(s_k-\sigma)}$.
Then
\eqref{dEs} guarantees the existence of $c,C$ such that
\begin{equation} \label{estWs22}
\begin{aligned}
c\int_{\sigma_1}^{\sigma_2}&\int_{\R^n}|W_s|^2(y,s)\rho(y)\,dy\,ds \\
&\leq C_{\sigma_1}{\mathcal E}(\sigma_1)-C_{\sigma_2}{\mathcal E}(\sigma_2)
+C\int_{\sigma_1}^{\sigma_2}\int_{\R^n}|W|^2(y,s)\rho(y)\,dy\,ds.
\end{aligned}
\end{equation}
Similarly, \eqref{GK1} implies
\begin{equation} \label{GK1int}
\begin{aligned}
4\int_{\sigma_1}^{\sigma_2}E(s)\,ds &= 
\int_{\R^n}|W|^2(y,\sigma_1)\rho(y)\,dy - \int_{\R^n}|W|^2(y,\sigma_2)\rho(y)\,dy  \\
 &\ \ +2\int_{\sigma_1}^{\sigma_2}
\phi^{-(p+1)}(s)\int_{\R^n}[f(\phi W)\cdot\phi W-2F(\phi W)](y,s)\rho(y)\,dy\,ds,
\end{aligned}
\end{equation}
hence
\eqref{F2} and $F\geq0$ imply
\begin{equation} \label{Esigma2}
4\int_{\sigma_1}^{\sigma_2}E(s)\,ds \geq -\int_{\R^n}|W|^2(y,\sigma_2)\rho(y)\,dy.
\end{equation}
Since \eqref{F4} and $F(U)=\int_0^1f(tU)\cdot U\,dt$
imply 
$$ |f(\phi W)\cdot\phi W|+F(\phi W)\leq C(|\phi W|^{p+1}+G(\phi W)|\phi W|)
= C(\phi^{p+1}|W|^{p+1}+\phi^{q+1}G(W)|W|),$$
estimate \eqref{GK1int} yields
\begin{equation} \label{Esigma1}
\begin{aligned}
4\int_{\sigma_1}^{\sigma_2}E(s)\,ds &\leq \int_{\R^n}|W|^2(y,\sigma_1)\rho(y)\,dy \\
&\ +C\int_{\sigma_1}^{\sigma_2}\int_{\R^n}(|W|^{p+1}+\phi^{q-p}G(W)|W|)(y,s)\rho(y)\,dy\,ds.  
\end{aligned}
\end{equation}
In the same way as in the proof of Theorem~\ref{thmUK} we obtain
\eqref{estK}. Since \eqref{F4} also implies
$$\begin{aligned} \frac{d}{ds}&\int_{\R^n}\xi\cdot W(y,s)\rho(y)\,dy+\beta\int_{\R^n}\xi\cdot W(y,s)\rho(y)\,dy \\
&\geq c\phi^{-p}(s)\int_{\R^n}G(\phi W)(y,s)\rho(y)\,dy
=c\phi^{q-p}(s)\int_{\R^n}G(W)(y,s)\rho(y)\,dy, 
\end{aligned} $$
the first inequality in \eqref{estK} guarantees
\begin{equation} \label{estK2}
\int_{\sigma_1}^{\sigma_2}\phi^{q-p}(s)\int_{\R^n}G(W)(y,s)\,dy\,ds\leq C. 
\end{equation}
The second inequality in \eqref{estK} and \eqref{bound-w2} also imply
\begin{equation} \label{estW2}
\int_{\sigma_1}^{\sigma_2}\int_{\R^n}|W|^2(y,s)\rho(y)\,dy\,ds \leq Ck^{\beta(2-p)_+}.
\end{equation}
Estimates \eqref{estK}, \eqref{estK2} and \eqref{Esigma1} guarantee
\begin{equation} \label{Esigma1a}
\int_{\sigma_1}^{\sigma_2}E(s)\,ds \leq C\sup_{s\in[\sigma_1,\sigma_2],\,y\in\R^n} |W|(y,s).
\end{equation}
Notice also that if $\sigma_2\leq s_k-1$, then
\begin{equation} \label{ck}
\sup_{s\in[\sigma_1,\sigma_2],\,y\in\R^n}|W|(y,s)\leq Cc_k k^\beta,
\end{equation}
where 
$$c_k:=\sup_{t\leq k(1-e),\,x\in\R^n}|U(x,t)|.$$

Estimates \eqref{Esigma1a} and \eqref{ck} guarantee the existence of $\hat\sigma_1\in[s_k-2,s_k-1]$
such that
\begin{equation} \label{calE1}
 {\mathcal E}(\hat\sigma_1)\leq E(\hat\sigma_1)\leq \int_{s_k-2}^{s_k-1}E(s)\,ds
  \leq Cc_k k^\beta. 
\end{equation}
Similarly, by using \eqref{estW2} we find $\tilde\sigma_2\in[s_k+1/2,s_k+1]$ such that
$$ \int_{\R^n}|W|^2(y,\tilde\sigma_2)\rho(y)\,dy
 \leq 2\int_{s_k+1/2}^{s_k+1}\int_{\R^n}|W|^2(y,s)\rho(y)\,dy\,ds \leq Ck^{\beta(2-p)_+}.$$
This estimate and \eqref{Esigma2}, \eqref{estW2} imply
$$\begin{aligned}
\int_{s_k}^{\tilde\sigma_2}{\mathcal E}(s)\,ds &=
\int_{s_k}^{\tilde\sigma_2}\Bigl(E(s)-\frac{C_E}2\int_{\R^n}|W|^2(y,s)\rho(y)\,dy\Bigr)\,ds \\
&\geq -\frac14\int_{\R^n}|W|^2(y,\tilde\sigma_2)\rho(y)\,dy
-Ck^{\beta(2-p)_+} \geq -Ck^{\beta(2-p)_+},
\end{aligned}$$
hence there exists $\hat\sigma_2\in[s_k,\tilde\sigma_2]$ such that
\begin{equation} \label{calE2}
{\mathcal E}(\hat\sigma_2)\geq -Ck^{\beta(2-p)_+}.
\end{equation}
Now \eqref{estWs22}, \eqref{calE1}, \eqref{calE2} and \eqref{estW2} yield
\begin{equation}\label{estWs23}
\int_{s_k-1}^{s_k}\int_{\R^n}|W_s|^2(y,s)\rho(y)\,dy\,ds\leq
\int_{\hat\sigma_1}^{\hat\sigma_2}\int_{\R^n}|W_s|^2(y,s)\rho(y)\,dy\,ds
\leq C(c_k k^\beta+k^{\beta(2-p)_+}).
\end{equation}

Now we proceed as in \cite{Q20}.
Taking $\lambda_k:=k^{-1/2}$ and setting
$$ V_k(z,\tau):=\lambda_k^{2\beta}W_k^0(\lambda_k z,\lambda_k^2\tau+s_k),
 \qquad z\in\R^n,\ -k\leq\tau\leq0,  $$ 
we obtain
$|V_k|+|\nabla V_k|\leq C$, $V_k(0,0)=U(0,0)\ne0$,
$$ \frac{\partial V_k}{\partial\tau}-\Delta V_k-(\lambda_k^{2\beta}\phi^{-1})^pf(\lambda_k^{-2\beta}\phi V_k)
  =-\lambda_k^2\Bigl(\frac12 z\cdot\nabla V_k+\beta V_k\Bigr). $$
Notice that $\lambda_k^{-2\beta}\phi(s)=\phi(s-s_k)=e^{\beta\tau/k}\to1$ for fixed $\tau$.
Using \eqref{estWs23} and $p\leq p_{sg}$ (i.e.~$\beta\leq\mu=2\beta-\frac{n-2}2$)
we also have
\begin{equation} \label{estVs}
\begin{aligned}
\int_{-k}^0\int_{|z|<\sqrt{k}}
 &\Big|\frac{\partial V_k}{\partial\tau}(z,\tau)\Big|^2\,dz\,d\tau
  =\lambda_k^{2\mu}
 \int_{s_k-1}^{s_k}\int_{|y|<1} 
\Big|\frac{\partial W_k^0}{\partial s}(y,s)\Big|^2\,dy\,ds \\
&\leq Ck^{-\mu}(c_k k^\beta+k^{\beta(2-p)_+}) 
\leq C k^{-\mu+\beta}\leq C \quad\hbox{as }\ k\to\infty.
\end{aligned}
\end{equation}
Since 
$$V_k(z,\tau)=e^{-\beta\tau/k}U(e^{-\tau/2k}z,k(1-e^{-\tau/k}))\to U(z,\tau)\ \hbox{ as }k\to\infty, $$
\eqref{estVs} implies
$\int_{-\infty}^0\int_{\R^n}|U_t|^2\,dx\,dt\leq C$.
Consequently, $\int_{-j-1}^{-j}\int_{\R^n}|U_t|^2\,dx\,dt\to0$ as $j\to\infty$,
and we can find $t_j\in[-j-1,-j]$ such that
$\int_{\R^n}|U_t|^2(x,t_j)\,dx\to0$ as $j\to\infty$.

First we show that $\|U(\cdot,t_j)\|_\infty\to 0$ as $j\to\infty$.
Assume on the contrary $\|U(\cdot,t_j)\|_\infty\geq\eps$
and choose $x_j$ such that $|U(x_j,t_j)|>\eps/2$.
Set $\tilde U_j(x):=U(x_j+x,t_j)$. Then $\{\tilde U_j\}$ is relatively compact in $C^2_{loc}$,
hence a subsequence of $\{\tilde U_j\}$ converges to $\tilde U\in{\mathcal K}$
satisfying $|\tilde U(0)|\geq\eps/2$, $-\Delta\tilde U=f(\tilde U)$ in $\R^n$,
which yields a contradiction.

Since $m_j:=\|U(\cdot,t_j)\|_\infty\to 0$ as $j\to\infty$, $|t_j-t_{j-1}|\leq2$ and $|f(U)|\leq C$, 
we have $\|U(\cdot,t)\|_\infty\leq e^{2C}m_j$ for $j\in[t_j,t_{j-1}]$, hence
$\sup_{t\leq t_j}\|U(\cdot,t)\|_\infty\to0$ as $j\to\infty$, which implies $c_k\to0$ as $k\to\infty$.
Consequently, $k^{-\mu}(c_k k^\beta+k^{\beta(2-p)_+})\to0$ as $k\to\infty$
so that \eqref{estVs} implies $\int_{-\infty}^0\int_{\R^n}|U_t|^2\,dx\,dt=0$,
hence $-\Delta U=f(U)$ in $\R^n$,
which contradicts our assumption and concludes the proof.

\section{Proof of Theorem~\ref{thm-proportional}} 
\label{sec-proof4}

In the proof of Theorem~\ref{thm-proportional} we will use the following lemma
which is a modification of \cite[Proposition~4]{Q16a}.

\begin{lemma} \label{lemQ-JDE}
Let $h\in C([0,\infty))$, $h(s)>0$ for $s>0$, 
$w\in W^{2,1;1}_{loc}\cap L^\infty(X\times\R)\cap C(\overline X\times\R)$, 
$w=0$ on $\partial\R^n_+$ if $X=\R^n_+$, and 
$w_t-\Delta w\leq-h(w)$ a.e.~in $W^+:=\{(x,t)\in X\times\R:w(x,t)>0\}$.
Then $W^+=\emptyset$.
\end{lemma}

\begin{proof}
Assume on the contrary that $W^+$ is nonempty. Then
$$W:=\sup_{X\times\R}w >0\qquad \hbox{and}\qquad H:=\inf_{W\geq s\geq W/4} h(s)>0.$$ 
Fix $\eps\in(0,H/(1+2n))$
and $(x^*,t^*)\in X\times\R$ such that $w(x^*,t^*)>W/2$.
Set
$$ w_\eps(x,t):=w(x,t)-\eps|x-x^*|^2
  -\eps\bigl(\sqrt{(t-t^*)^2+1}-1\bigr),\qquad x\in X,\ t\in\R,$$
and $z_k:=w_\eps*\varphi_k$, where $\varphi_k$ is a sequence of mollifiers
($\varphi_k(x,t)=k^{n+1}\varphi(kx,kt)$, $\varphi\geq0$ is
smooth, with support in the unit ball of $\R^n\times\R$,
$\int_{\R^n\times\R}\varphi=1$).
If $k$ is large enough,
then $z_k$ attains its supremum at some $(x_k,t_k)$
and $w\ge w_\eps>W/4$ on $B:=B_{1/k}(x_k,t_k)\subset X\times\R$
(due to the uniform continuity of $w_\eps$ in the compact 
$\{(x,t)\in\overline X\times\R:w_\eps(x,t)\ge 0\}$).
Notice also that the perturbation term 
$\psi(x,t):=-\eps|x-x^*|^2
  -\eps\bigl(\sqrt{(t-t^*)^2+1}-1\bigr)$ satisfies 
$$ (\psi_t-\Delta\psi)(x,t)=2\eps
n-\eps\frac{t-t^*}{\sqrt{(t-t^*)^2+1}}\le\eps(1+2n),$$
hence $(\psi_t-\Delta\psi)*\varphi_k\le\eps(1+2n)$.
Consequently,
$$ \begin{aligned}
  0 &\leq [(z_k)_t-\Delta z_k](x_k,t_k) \\
   &= [(w_\eps)_t-\Delta w_\eps]*\varphi_k(x_k,t_k) \\
  &\leq [w_t-\Delta w]*\varphi_k(x_k,t_k)+\eps(1+2n) \\ 
 &\leq -h(w)*\varphi_k(x_k,t_k)+\eps(1+2n) \\
  &\leq -H+\eps(1+2n)<0 
\end{aligned} $$
which yields a contradiction.
\end{proof}

\begin{proof}[Proof of Theorem~\ref{thm-proportional}]
Set $w:=u-v$. Then $|w|\leq M$ for some $M>0$. 
We will prove $w\leq0$ (the proof of $w\geq0$ follows
by replacing the role of $u$ and $v$).
Notice that $w$ is a solution of
$$ 
w_t-\Delta w =\phi(u,v)(\psi_1(u,v)-\mu\psi_2(u,v)), \quad\hbox{where}\quad
\left\{\ \begin{aligned}
\psi_1(u,v)&:=k(u)g(v)-k(v)g(u),\\ 
\psi_2(u,v)&:=kg(u)-kg(v).
\end{aligned}\right. $$
Set $\tilde h(t):=\inf_{\tau\geq0}(\varphi(\tau+t)-\varphi(\tau))$, $t\in(0,M]$.
Then 
$$\begin{aligned}
\tilde h(t) &= \min(\inf_{\tau\geq t/2}(\varphi(\tau+t)-\varphi(\tau)),\inf_{\tau<t/2}(\varphi(\tau+t)-\varphi(\tau))) \\
 &\geq \min(t\min_{s\in[t/2,M]}\varphi'(s),\varphi(t)-\varphi(t/2))=:h(t).
\end{aligned}$$
Set also $h(0)=0$. Then $h$ is continuous and $h(t)>0$ for $t\in(0,M]$.

Assume on the contrary that $w>0$ in some nonempty domain $W^+\subset X\times\R$.
The following estimates are true in $W^+$.

First assume $\mu>0$ (hence $\varphi=kg$ and $\tilde k=k$). Then
$$\psi_1(u,v)=
\begin{cases} -k(0)g(u)\leq0 & \hbox{ if }v=0, \\
    g(u)g(v)(k/g(u)-k/g(v))\leq0 &\hbox{ if }v>0,
\end{cases}$$
and we also have
$$\psi_2(u,v)=\varphi(u)-\varphi(v)=\varphi(v+w)-\varphi(v)\geq\tilde h(w)\geq h(w),$$ 
hence
$w_t-\Delta w\leq-\phi(u,v)\mu\psi_2(u,v) \leq-c_\phi\mu h(w)$.
Now Lemma~\ref{lemQ-JDE} implies $w\leq0$.

Next assume $\mu=0$ (hence $\varphi=g$, $\eps>0$ and $\tilde k=k-\eps$).
Then
$$ \tilde k(u)g(v)-\tilde k(v)g(u)=
\begin{cases}
-\tilde k(0)g(u)\leq 0 &\hbox{ if $v=0$}, \\
g(u)g(v)(\tilde k/g(u)-\tilde k/g(v))\leq 0 &\hbox{ if $v>0$},
\end{cases}
$$
hence
$$\psi_1(u,v) \leq \eps(g(v)-g(u))=-\eps(\varphi(u)-\varphi(v))\leq-\eps h(w).$$
Consequently, $w_t-\Delta w\leq-c_\phi\eps h(w)$ 
and Lemma~\ref{lemQ-JDE} implies $w\leq0$.
\end{proof}

\section{Proof of Theorem~\ref{thmUB}} 
\label{sec-proof5}

The nonexistence of nontrivial bounded entire solutions of $V_t-\Delta V=\varphi(V)$ implies $\varphi(U)\ne0$ whenever $U\ne0$.
In particular, if $\varphi=f$, then $f^+(\lambda)>0$ for $\lambda>0$.

If $U\ne0$, then we set 
$$M(U):=\sqrt{\frac{f^+(|U|)}{\sigma+|U|}},$$
where $\sigma=1$ in Case (i) and $\sigma=0$ otherwise.
Assuming $|U|\geq\Lambda$ in Case (i) and $|U|\leq\Lambda$ in Case (ii),
we will prove
\begin{equation} \label{estM}
 M(U)\leq C(\sigma+\hbox{\rm dist}^{-1}(X,\partial D)).
\end{equation}
Notice that \eqref{estM} implies estimates \eqref{UBh} and \eqref{UB}.

\cite[Remark 7.1]{QS24} 
implies the existence of $c>0$ and $r>1$ such that
$$f^+(\lambda s)\geq cs^r f^+(\lambda)
\quad\hbox{whenever }s\geq1\hbox{ and }
\begin{cases} \lambda\geq\Lambda &\hbox{ in Case (i)},\\
              \lambda s\in(0,\Lambda] &\hbox{ in Case (ii)},\\
              \lambda>0 &\hbox{ in Case (iii)}.
\end{cases}$$
Fix $s_0>1$ such that $\kappa:=cs_0^{r-1}>1$.
Then the inequality $f^+(\lambda s)\geq cs^r f^+(\lambda)$
implies 
\begin{equation} \label{fsup2}
f^+(\lambda s)\geq \kappa s f^+(\lambda)\quad\hbox{for }s\geq s_0
\end{equation}
and inequality \eqref{fsup2} implies
\begin{equation} \label{s0bound}
M^2(U)\leq \kappa M^2(V)\ \Rightarrow\ |U|\leq s_0|V|,
\end{equation}
provided $|U|,|V|\geq\Lambda$ in Case (i)
and $|U|,|V|\leq\Lambda$ in Case (ii). 

We will use the same arguments as in the proof of \cite[Theorem 4.3]{S23}. 
Assume to the contrary that there exist nonempty open sets $D_k$, 
solutions $U_k$ of \eqref{eq-U} in $D_k$ and $X_k=(x_k,t_k)\in D_k$ such that
\begin{equation} \label{flarge}
 M_k(X_k)>\frac{\sqrt\kappa k}{(\sqrt{\kappa}-1)}({\sigma+}\hbox{dist}^{-1}(X_k,\partial D_k)),
\end{equation}
where $M_k:=M(U_k)$,  
$|U_k(X_k)|\geq\Lambda$ in Case (i) and $|U_k|\leq\Lambda$ in Case (ii).
An obvious modification of the doubling lemma in \cite{PQS07a}
guarantees that we may assume
$M_k(X)\leq\sqrt\kappa M_k(X_k)$ in $\hat D_k=\{X:d(X,X_k)\leq \frac{k}{M_k(X_k)}\}$.
(If $k$ is large, then the inequality $|U_k(X_k)|\geq\Lambda$ in Case~(i) remains true
after the modification of $X_k$ in the doubling lemma, since $M_k\to\infty$ 
in Case~(i) and $f^+(U)$ is bounded for $|U|\leq\Lambda$.)
Consequently, \eqref{s0bound} implies $|U_k(X)|\leq s_0|U_k(X_k)|$ for $X\in \hat D$.
Set
$$ m_k:=|U_k(X_k)|,\quad\alpha_k:=\frac1{M_k(X_k)},\quad
V_k(Y)=\frac1{m_k}U_k(x_k+\alpha_k y,t_k+\alpha_k^2 s), $$
where $Y=(y,s)\in\tilde D_k:=\{Y:|y|<k/2,\ |s|<k^2/4\}$.
Then $|V_k(0,0)|=1$, $|V_k|\leq s_0$ and
\begin{equation} \label{eqVk}
 (V_k)_s-\Delta V=\frac{\alpha_k^2}{m_k}f(m_k V_k)=\frac{\sigma+m_k}{m_k}\frac1{f^+(m_k)}f(m_k V_k)\quad\hbox{in }\tilde D_k.
\end{equation}
Since $m_k\to\infty$ in Case~(i) and $m_k\leq\Lambda$ in Case~(ii),
the right-hand side of \eqref{eqVk} is bounded,
and we can suppose that $V_k\to W$ locally uniformly,
where $|W(0,0)|=1$ and $|W|\leq s_0$.
If $m_k\to\infty$ or $m_k\to0$, then $W$ is a nontrivial bounded entire solution 
of $W_s-\Delta W=\varphi(W)$, where $\varphi=\finfty$ or $\varphi=\fzero$, respectively,
which yields a contradiction.
Hence we may assume $m_k\to a\in(0,\infty)$.
Then $W$ is a nontrivial bounded entire solution of 
$W_s-\Delta W=\frac{f(aW)}{f^+(a)}$, hence $\tilde W(y,s):=aW(\lambda y,\lambda^2 s)$
with $\lambda:=\sqrt{f^+(a)/a}$ is a nontrivial bounded entire solution of 
$\tilde W_s-\Delta\tilde W=f(\tilde W)$, which yields a contradiction again.

\vskip3mm
\noindent{\bf Acknowledgement.}
The first author was supported in part by the Slovak Research and Development Agency
under the contract No. APVV-23-0039 and by VEGA grant 1/0245/24.
Part of this work was done during visits of the first author at Universit\'e Sorbonne Paris Nord
and of the second author at Comenius University. The authors thank these institutions for their support.


\end{document}